\documentclass[11pt,a4paper,reqno]{amsart}
\usepackage{epsfig,amssymb,latexsym,amscd,amsmath,amsthm,stmaryrd}
\usepackage{verbatim}

\usepackage[all,2cell]{xy}
\xyoption{v2}
\UseAllTwocells

\theoremstyle{plain}

\newtheorem{teo}{Theorem}[section]
\newtheorem{theo}[teo]{Theorem}
\newtheorem{coro}[teo]{Corollary}
\newtheorem{lema}[teo]{Lemma}
\newtheorem{prop}[teo]{Proposition}

\theoremstyle{remark}

\newtheorem{rema}[teo]{Remark}

\theoremstyle{definition}

\newtheorem{defi}[teo]{Definition}
\newtheorem{obse}[teo]{Observation}

\newtheorem{example}[teo]{Example}

\newcommand{\id}{\ensuremath{\mathrm{id}}}

\newcommand{\Hom}{\ensuremath{\mathrm{Hom}}}
\renewcommand{\Im}{\ensuremath{\mathrm{Im}}}
\newcommand{\Ker}{\ensuremath{\mathrm{Ker}}}

\newcommand{\dem}{\emph{Proof:}\ }
\newcommand{\ant}{\mathcal{S}}

\newcommand{\K}{{\mathbb K}}

\newcommand{\Z}{{\mathbb Z}}

\newcommand{\C}{{\mathbb C}}

\newcommand{\tl}{\triangleleft}
\newcommand{\tr}{\triangleright}
\newcommand{\trb}{\blacktriangleright}

\newcommand{\fin}{\qed}

\newcommand{\cop}{\ensuremath{\mathrm{cop}}}

\newcommand{\cc}{\left(}
\newcommand{\dd}{\right)}

\begin{document}

\title[Some constructions of compact quantum groups.]
{Some constructions of compact quantum groups.}
\author{Andr\'es Abella}
\address{Facultad de Ciencias\\Universidad de la Rep\'ublica\\
Igu\'a 4225\\11400 Montevideo\\Uruguay\\}
\email{andres@cmat.edu.uy}
\author{Walter Ferrer Santos}
\thanks{The second author would like to thank Csic-UDELAR,  Conycit-MEC, Uruguay and Anii, Uruguay.}
\email{wrferrer@cmat.edu.uy}
\author{Mariana Haim}
\thanks{The third author would like to thank Conycit-MEC, Uruguay and SNI, Uruguay.}
\email{negra@cmat.edu.uy}

\begin{abstract}
The purpose of this paper is to consider some basic constructions in the category of compact quantum
groups --for example the case of extensions-- with special emphasis in the finite dimensional situation.
We give conditions, in some cases necessary and sufficient, to extend to the new objects the original compact structure.
\end{abstract}
\maketitle

\section{Introduction}

Since the beginnings of group theory, in the work of Frobenius, Schur and others, the concept of product and its variations
as well as the concept of ``extension'' have been recognized as important tools in order to understand the structure of
the groups and of its representations. This line of thought has been carried through many different algebraic situations, more recently to the theory of Hopf algebras, see for example \cite{kn:andrus-devoto}, \cite{kn:andrus2}, \cite{kn:Majid},\cite{kn:masuoka},  \cite{kn:masuoka2}, \cite{kn:schneider2} and  \cite{kn:takeuchi}. This paper intends to understand how to carry on the classical
constructions of products and extensions to the context of compact quantum groups. In this direction, previous results
have been obtained in the particular settings of Kac algebras or von Neumann algebras, see for example \cite{kn:Majid2}, and \cite{kn:Yamanouchi}, and also \cite{kn:andrus2} for a Hopf algebra approach.

We assume the reader to be familiar with the basic results and notations in Hopf theory, specially in the case of
 compact quantum groups. We refer to \cite{kn:mont} for the general theory and notations and to
\cite{kn:andres-walter-mariana},  \cite{kn:andrus0} and \cite{kn:koor2} for the case of compact quantum groups.

Next we present a brief summary of the contents of this paper.

In Section \ref{section:prelim}, we recall the basic definitions of $*$--Hopf algebras, its representations and
corepresentations and the notion of unitary inner product. We list some basic well known results  that will be
used throughout the paper.

In Section \ref{section:compactness}, we study the behaviour of the compactness property of a $*$--Hopf algebra when we perform a Drinfel'd twist.

In Section \ref{section:matchedpairs}, we recall the definition of matched pair of Hopf algebras and the naturally obtained product, and find conditions
to extend the compact structures on the factors to a compact structure on the product. We consider the particular case
of the quantum double, and the characterization of compactness in terms of the category of Yetter--Drinfel'd modules.

In Section \ref{section:extensions}, we study the case of extensions of compact quantum groups, and find conditions to
extend the compact structure from the given Hopf algebras to the extension.

In Section \ref{section:singerpairs} we consider the particular cases of (cocycle )Singer pairs and of pairs of groups.

The authors would like to thank the referee of this paper for many helpful comments, in particular for pointing out the results of \cite{kn:andrus2} that suggested some generalizations of the formulations and results of a first version of the paper.

\section{Preliminaries and basic notations}\label{section:prelim}

All the objects considered will be based on vector spaces over the field of complex numbers $\C$. If $V$ is a vector space, its linear dual will be denoted as $V^\vee$.
We use the symbol $\star$ for the convolution product of morphisms and if $f$ is such a morphism, $f^{-1}$ represents its (eventually existing) inverse. We adopt Sweedler's notation for the comultiplication and for the comodule structures.

\subsection{Finite Hopf algebras}

If $H$ is a Hopf algebra a function $\varphi\in H^{\vee}$ is called a \emph{right} or \emph{left integral} if
$\sum\varphi(x_1)x_2=\varphi(x)1$ or $\sum x_1\varphi(x_2)=\varphi(x)1$, for all
$x\in H$, according to the case.
A \emph{normal integral} is a left and right integral $\varphi$ such that $\varphi(1)=1$; in case it exists, a normal integral is unique.
If a right or left integral $\varphi$ verifies
$\varphi(1) = 1$, then it is a right and left integral and
hence it is normal, see \cite{kn:Dascalescu}.

A \emph{finite} Hopf algebra is a finite dimensional Hopf algebra.
A Hopf algebra $H$ is \emph{semisimple} if every $H$-module is completely reducible and
it is \emph{cosemisimple} if every $H$-comodule is completely reducible.
A semisimple Hopf algebra is always finite.

\begin{teo}[\cite{kn:Dascalescu}, Theorem 7.4.6] \label{teo:semisimple}\label{teo:props-ss}
Let $H$ be a Hopf algebra. The following assertions are equivalent:
\begin{enumerate}
\item
$H$ is cosemisimple.
\item \label{intH^vee}
There exists a normal integral $\varphi\in H^{\vee}$.
\end{enumerate}
If $H$ is a finite Hopf algebra, then the following assertions are equivalent:
\begin{enumerate}
\item \label{ss}
$H$ is semisimple.
\item \label{intH}
There exists a unique element $t\in H$ --called a normal integral-- such that $\varepsilon(t)=1$ and $tx=xt=\varepsilon(x)t$,
for all $x\in H$.
\item \label{coss}
$H$ is cosemisimple.
\item
The antipode of $H$ is an involution i.e. $\ant^2=\id$.
\end{enumerate}
Moreover if $H$ is a semisimple Hopf algebra, and $t\in H$ and $\varphi\in H^{\vee}$ are as above, then

\begin{enumerate}
\item
$\ant(x)=(\dim H)\sum \varphi(t_1x)t_2$, for all $x\in H$.
\item
$H\rightharpoonup\varphi=H^{\vee}=\varphi\leftharpoonup H$, where
$(x\rightharpoonup\varphi)(y)=\varphi(yx)$ and
$(\varphi\leftharpoonup x)(y)=\varphi(xy)$, for all $x,y\in H$.
\item
$\ant(t)=t$ and $\varphi\ant=\varphi$.
\item
$\sum t_1\otimes t_2=\sum t_2\otimes t_1$ and $\varphi(xy)=\varphi(yx)$, for all
$x,y\in H$.
\qed
\end{enumerate}
\end{teo}

\begin{rema} [\cite{kn:bbt}]
Let $H$ be a cosemisimple Hopf algebra and $\varphi\in H^\vee$ its normal integral.
There exists a unique algebra automorphism $\mathcal N:H\to H$ characterized by the equality
\begin{equation}\label{eq:naka}
\varphi(xy)=\varphi(y\mathcal N(x)), \quad \forall x,y\in H
\end{equation} and called the \emph{Nakayama automorphism}
and an algebra homomorphism $\alpha\in H^\vee$ --called the \emph{modular function}-- such that
\begin{equation} \label{eq:nakayama}
\mathcal N(x)=\sum \ant^2(x_1)\alpha(x_2),
\quad
\forall x\in H.
\end{equation}
It is easy to deduce from the above equation that the inverse of $\mathcal N$ satisfies a similar equation:

\begin{equation} \label{eq:-nakayama}
  \mathcal N^{-1}(x)=\sum \ant^{-2}(x_1)(\alpha \ant)(x_2),
\quad
\forall x\in H.
\end{equation}
\end{rema}

\subsection{Compact quantum groups}\label{cqg}

A \emph{$*$-bialgebra} is a pair $(H,*)$, where $H$ is a bialgebra, and
$H\stackrel{*}{\to}H$ is a conjugate linear involution which is antimultiplicative and
comultiplicative. It follows from the definition that $1^*=1$ and $\varepsilon*=\overline{\varepsilon}$.
A \emph{$*$-Hopf algebra} is a $*$-bialgebra $(H,*)$ such that $H$ is a Hopf algebra; in this situation it
follows that $(\ant *)^2=\id$.
A $*$-Hopf algebra morphism between $*$-Hopf algebras is a Hopf algebra map  that commutes with  the $*$--operators.

If $H$ is a $*$-Hopf algebra, and $V$ is a right $H$-comodule, a sesquilinear form
$\langle\,\, , \,\rangle:V\times V\to\C$ is \emph{invariant} if
\begin{equation} \label{eq:pi-*-invariante}
\sum \langle u_0,v\rangle \ant(u_1)=\sum \langle u,v_0\rangle v_1^*,\quad
\forall u,v\in V.
\end{equation}
A \emph{unitary (right) comodule} is a pair $(V,\langle\,\, ,\, \rangle)$, where $V$
is a right $H$-comodule and $\langle\,\, ,\, \rangle$ is an invariant inner product on $V$.

In the case of left comodules, the definition is similar and the unitary condition becomes:
\begin{equation}\label{eq:pi-*-invarianteleft}
\sum\ant^{-1}(u_{-1})\langle u_0,v\rangle = \sum v_{-1}^*\langle u,v_0\rangle,\quad
\forall u,v\in V.
\end{equation}

If $V$ is a unitary comodule and $W\subset V$ is a subcomodule --right or left--, then the orthogonal complement $W^\perp$
is also a subcomodule. Hence a unitary comodule is always completely reducible.
\medskip

A \emph{compact quantum group} (CQG) is a $*$-Hopf algebra $H$ such that every right
$H$-comodule admits an invariant inner product
(see \cite{kn:andres-walter-mariana}, \cite{kn:andrus0}, \cite{kn:koor2}, \cite{kn:woro}).
\medskip

If $H$ is a CQG, then it is cosemisimple 
and admits an inner product $\langle\,\, ,\, \rangle$ such that for all $x,y\in H$,
\begin{align*}
\sum \langle x_1,y\rangle \ant(x_2) &= \sum \langle x,y_1\rangle y_2^* .
\end{align*}
Moreover, as every cosemisimple $*$-Hopf algebra $H$ admits a normal integral $\varphi\in H^{\vee}$, and
this integral satisfies that $\varphi\ant=\varphi$ and $\varphi *=\overline{\varphi}$,
then the formula
\begin{equation*} 
\langle x,y \rangle_\varphi=\varphi(y^*x),\quad \forall x,y\in H.
\end{equation*}
defines an invariant Hermitian form.
A direct computation shows that:
\begin{equation} \label{H-*-representacion-regular}
\langle zx,y \rangle_\varphi=\langle x,z^*y \rangle_\varphi,\quad
\forall x,y,z\in H.
\end{equation}
Moreover, if $H$ is finite, then $\varphi(xy)=\varphi(yx)$ implies
\begin{equation} \label{H-*-representacion-regular2}
\langle xz,y \rangle_\varphi=\langle x,yz^* \rangle_\varphi,\quad
\langle x^*,y^* \rangle_\varphi=\langle y,x \rangle_\varphi,\quad
\forall x,y,z\in H.
\end{equation}

\begin{teo}[\cite{kn:andrus}, \cite{kn:woro}] \label{teo:andrus1}
A cosemisimple $*$-Hopf algebra $H$ is a CQG if and only if
$\langle\,\, ,\, \rangle_\varphi$ is an inner product, i.e. if $\varphi(x^*x)>0, \forall x\neq 0, x\in H$. \qed
\end{teo}

\begin{rema}\label{rk:nos}
A Hopf subalgebra of a cosemisimple $*$-Hopf algebra is again a $*$-Hopf algebra. Indeed, if we write the decomposition of $H$ into simple coalgebras $H=\bigoplus_{\lambda \in \Lambda} C_\lambda$, we have (see for example \cite{kn:andres-walter-mariana}) that $S(C_\lambda)=(C_\lambda)^*$ for each $\lambda \in \Lambda$. Taking a Hopf-subalgebra $K$ of $H$, it will be of the form $K=\bigoplus_{\lambda \in \Lambda} C_\lambda$ and therefore $K^* =\bigoplus_{\lambda \in \Lambda} C_\lambda^*=S(K)=K$.
\end{rema}
In view of Theorem \ref{teo:andrus1} and Observation \ref{rk:nos}, we have the following result.

\begin{coro}\label{coro:cqg-cerito}
If $H$ is a CQG, then every sub-Hopf algebra of $H$ is a CQG. \qed
\end{coro}

\begin{rema}
If $(H,*)$ is a $*$-Hopf algebra, then it is easy to see that $(H^{\text{op}},*)$,
$(H^{\text{cop}},*)$ and $(H^{\text{bop}},*)$ are $*$-Hopf algebras.
Using Theorem \ref{teo:andrus1} we prove that if $H$ is a CQG, then  $H^{\text{op}}$,
$H^{\text{cop}}$ and $H^{\text{bop}}$ are also CQG.
\end{rema}

Let $H$ be a cosemisimple $*$-Hopf algebra and let $\varphi\in H^{\vee}$ be the normal
integral. As $\varphi$ is also a normal integral for $H^{\text{cop}}$, then $\langle\,\, ,\, \rangle_\varphi$
is invariant in $H^{\text{cop}}$, \textit{i.e.} for all $x,y\in H$,
\begin{align}
\sum\ant^{-1}(x_{1})\langle x_2,y\rangle_\varphi = \sum y_{1}^*\langle x,y_2\rangle_\varphi, \label{eq:left-inv-H}
.
\end{align}
Let us assume that $H$ is a CQG and that $(V,\chi)$ is a left $H$-comodule.
Then $(V,\operatorname{sw}\chi)$ is naturally a right $H^{\text{cop}}$-comodule where $\operatorname{sw}$ is the
usual transposition of the tensor factors, hence --being $H^{\text{cop}}$ a CQG-- we deduce the existence a right
invariant inner product $\langle\,\, ,\, \rangle$ in $(V,\operatorname{sw}\chi)$, this inner product can be read as
left invariant inner product in the original $(V,\chi)$.

Then a $*$-Hopf algebra $H$ is a CQG if and only if every left $H$-comodule admits an invariant inner product.
Note that the equation \eqref{eq:left-inv-H}
implies that $\langle\,\, ,\, \rangle_\varphi$ is an invariant inner product for the regular left $H$-comodule $H$.

For future reference we write the following result.

\begin{teo}[\cite{kn:andrus}, Theorem 2.6] \label{teo:andrus2}
If $*$ and $\#$ are involutions in a Hopf algebra $H$ such that
$(H,*)$ and  $(H,\#)$ are CQG, then there exists a Hopf algebra automorphism
$T:H\to H$ such that $T*=\# T$.
\qed
\end{teo}

\subsection{Duality for finite compact quantum groups.}

Let $H$ be a finite CQG.
First we observe that the relations $\ant^2=\id$ and $\ant *\ant *=\id$, imply $\ant *=*\ant $.
\textit{i.e.} the antipode is a $*$-map.

In this situation the dual Hopf algebra $H^{\vee}$ has a conjugate linear involution defined by
\begin{equation} \label{eq:*-dual}
\alpha^*(x)=\overline{\alpha(\ant(x)^*)}= \overline{\alpha(\ant(x^*))},\quad \forall \alpha\in H^{\vee},\ x\in H.
\end{equation}
It is clear that $H^{\vee}$ equipped with this involution is a $*$-Hopf algebra and that
the functor $H\mapsto H^{\vee}$ is a  contravariant automorphism in the category of finite
$*$-Hopf algebras. Moreover, the cannonical injection $H \rightarrow H^{\vee\vee}$ is an isomorphism of $*$-Hopf algebras.

\begin{lema} \label{lema:dual-finite-cqg}
If $H$ is a finite $*$-Hopf algebra, then $H$ is a CQG if and only if $H^{\vee}$ is a CQG.
\end{lema}
\dem
From Theorem \ref{teo:semisimple} it follows that $H$ is semisimple if and only if $H^{\vee}$ is semisimple. Assume the semisimplicity hypothesis and denote $t \in H$ and $\varphi \in H^\vee$ as before.

If $n=\dim H$, then:
\begin{equation*}
\langle x\rightharpoonup\varphi, y\rightharpoonup\varphi\rangle_t=
\frac{1}{n}\langle x, y\rangle_\varphi,\quad \forall x,y\in H
\end{equation*}

Indeed,
\begin{align*}
\langle x\rightharpoonup\varphi, y\rightharpoonup\varphi\rangle_t
&=
\big( (y\rightharpoonup\varphi)^*(x\rightharpoonup\varphi) \big)(t)
=
\sum (y\rightharpoonup\varphi)^*(t_1)(x\rightharpoonup\varphi)(t_2)
=
\sum \overline{\varphi(\ant(t_1)^* y)}\, \varphi(t_2x) \\
&=
\sum \varphi\big((\ant(t_1)^* y)^*\big)\, \varphi(t_2x)
=
\sum \varphi(y^*\ant(t_1))\, \varphi(t_2 x)
=
\sum \varphi(y^*\ant(t_2))\, \varphi(t_1 x)\\
&=
\varphi \left (y^*\ant \left (\sum \varphi(t_1 x)t_2\right ) \right)
=
\varphi \left (y^*\ant\left (\frac{1}{n}\ant(x)\right )\right )
=
\frac{1}{n}\langle x, y\rangle_\varphi.
\end{align*}

Theorem \ref{teo:semisimple} guarantees that $H\rightharpoonup\varphi=H^{\vee}$, hence in view of the above equality it is
clear that $\langle -, - \rangle_t$ is positive definite if and only if the same holds for $\langle -, - \rangle_\varphi$,
and then Theorem \ref{teo:andrus1} yields the result.
\qed

\begin{example} \label{ex:grupos}
If $G$ is a finite group, then the group algebra $\C G$ is a CQG with
the structure $g^*=g^{-1}$, for all $g\in G$.
Hence the dual Hopf algebra $\C^G$ with dual basis $\{e_g:\ g\in G\}$
is a CQG with respect to $(e_g)^*=e_g$, for all $g\in G$.
Theorem \ref{teo:andrus2} implies that the above are the unique possible CQG structures in $\C G$ and $\C^G$.
\end{example}

\begin{rema}\label{rem:*representation}
Let $H$ be a finite $*$-Hopf algebra and $V$ a left $H$-comodule with structure $v \mapsto \sum v_{-1} \otimes v_0$. Then
$V$ is a left $H^{\vee}$-module  and the action and coaction are related by the equality:
$\ant(\alpha)\cdot v=\sum \alpha(v_{-1}) v_0$, for all $\alpha\in H^\vee$, $v\in V$.
If $\langle\,\, ,\, \rangle$  is an inner product on $V$, it is easy to see (using that $\ant$ is bijective)
that $(V,\langle\,\, ,\, \rangle)$ is an unitary $H$-comodule
if and only if the representation of $H^\vee$ in $V$ is a $*$-representation --the inner product in that case is said to
be \emph{invariant} for the action--, \textit{i.e.}
\begin{equation*}
\langle \alpha\cdot u ,v \rangle = \langle  u ,\alpha^*\cdot v \rangle,\quad \forall \alpha\in H^\vee,\ u,v\in V.
\end{equation*}
\end{rema}
From the above Observation and Lemma \ref{lema:dual-finite-cqg}, we conclude:

\begin{lema} \label{lema:equivalencias}
Let $H$ be a semisimple $*$-Hopf algebra. Then the following assertions are equivalent:
\begin{enumerate}
\item
$H$ is a CQG, \textit{i.e.} for every right (left) $H$-comodule $V$ there exists an inner product
$\langle\,\, ,\, \rangle$ on $V$ such that
\begin{align}
\sum \langle u_0,v\rangle \ant(u_1) &= \sum \langle u,v_0\rangle v_1^*,\quad
\forall u,v\in V.\label{eq:right-inv-V} \\
\sum\ant(u_{-1})\langle u_0,v\rangle &= \sum v_{-1}^*\langle u,v_0\rangle, \quad \forall u,v\in V.\label{eq:left-inv-V}
\end{align}

\item For every left (right) $H$-module $V$ there exists an inner product
$\langle\,\, ,\, \rangle$ on $V$ such that

\begin{align}
\langle x\cdot u ,v \rangle &= \langle  u ,x^*\cdot v \rangle, \quad \forall u,v\in V.\label{eq:*-rep} \\
\langle u\cdot x ,v \rangle &= \langle  u ,v\cdot x^* \rangle, \quad \forall u,v\in V.\label{eq:*-rep-right}
\end{align}
\qed
\end{enumerate}
\end{lema}
In the preceeding subsection we proved  that the validity of \eqref{eq:right-inv-V} or \eqref{eq:left-inv-V} for all
comodules are  equivalent conditions, and the same is true regarding \eqref{eq:*-rep} and \eqref{eq:*-rep-right}.
Moreover, these last two equalities  are consistent with the equalities appearing in
\eqref{H-*-representacion-regular} and \eqref{H-*-representacion-regular2}.

\begin{rema} \label{rem:inv-prod-tens}
If $H$ is a finite CQG and $V, W$ are right $H$-comodules with invariant inner products $\langle\,\, ,\, \rangle_V$ and
$\langle\,\, ,\, \rangle_W$, then it is clear that the inner product
defined on $V\otimes W$ by the formula:
\[
\langle v\otimes w,v'\otimes w'\rangle =\langle v,v'\rangle_V\langle w, w'\rangle_W
\]
is invariant with respect to the diagonal right $H$-comodule structure on $V\otimes W$.
Obviously an analogous result is valid for $H$--modules and also changing right for left.
\end{rema}
\section{Compactness of Drinfel'd twists}\label{section:compactness}

In this section we study how compactness behaves under performing a Drinfel'd twist to a $*$--Hopf algebra. We choose to develop the basic results in the case of twisting the product with a cocycle, the case of twisting the coproduct with a cococycle being similar.
\subsection{Twisting products}

Consider a Hopf algebra $\mathcal H = (H,\Delta,\varepsilon,m,u,\mathcal S)$ with invertible antipode.

Recall that a cocycle is a (convolution) invertible element $\chi: H \otimes H \rightarrow \K$ that satisfies the following properties. 

\begin{equation}\label{equation:cococycle1}
(\chi \otimes \varepsilon) \star \chi(m \otimes \operatorname{id}) = (\varepsilon \otimes \chi) \star \chi(\operatorname{id} \otimes m)
\end{equation}
\begin{equation}\label{equation:cococycle2}
\chi(1,x) = \chi(x,1)=\varepsilon(x)
\end{equation}

The fist equation above becomes:

\begin{equation}\label{equation:cococycle4}
\sum \chi(x_1,y_1)\chi(x_2y_2,z)=\sum \chi(y_1,z_1)\chi(x,y_2z_2)
\end{equation}

Taking inverses one easily deduces that:

\begin{equation}\label{equation:cococycle5}
\chi^{-1}(m \otimes \operatorname{id}) \star (\chi^{-1} \otimes \varepsilon)  = \chi^{-1}(\operatorname{id} \otimes m) \star (\varepsilon \otimes \chi^{-1})
\end{equation}
\begin{equation}\label{equation:cococycle6}
\chi^{-1}(1,x) = \chi^{-1}(x,1)=\varepsilon(x).
\end{equation}

Equation (\ref{equation:cococycle5}) above can easily be transformed into:

\begin{equation}\label{equation:cocycle8}
\sum \chi(x_1,y_1z_1)\chi^{-1}(x_2y_2,z_2)=\sum
\chi^{-1}(y_1,z)\chi(x,y_2).
\end{equation}

\begin{obse}
It is well known that if we let $\mathcal H_{\chi}= (H,\Delta,\varepsilon,m_\chi,u,\mathcal S_\chi)$ where $m_\chi=\chi \star m \star \chi^{-1}$  and $\mathcal S_\chi= \kappa \star \mathcal S \star \kappa^{-1}$ with $\kappa=\chi(\operatorname{id} \otimes \mathcal S)\Delta$ and $\kappa^{-1}(x)=\chi^{-1}(\mathcal S \otimes \operatorname{id})\Delta$, then  $\mathcal H_\chi$ is a Hopf algebra (see V. G. Drinfel'd,\cite{kn:dri}).
\end{obse}

More explicitly we have that: $m_\chi(x,y)= x \cdot_\chi y= \sum \chi(x_1,y_1)x_2y_2\chi^{-1}(x_3,y_3)$, \\$ \mathcal S_\chi(x)= \sum \chi(x_1,\mathcal Sx_2) \mathcal S(x_3)\chi^{-1}(\mathcal Sx_4,x_5)$, and $\kappa(x)= \sum \chi(x_1,\mathcal Sx_2), \kappa^{-1}(x)=\sum \chi^{-1}(\mathcal Sx_1,x_2)$.
\medskip

\subsection{Twisting $*$--Hopf algebras}

In this subsection we review well known results concerning twistings and $*$--structures on a Hopf algebra defined over $\C$.

We consider the case that we twist the product with a cocycle $\chi:H \otimes H \rightarrow \C$.

In this situation, we want to find conditions on $\chi$ that guarantee that the same $*$--structure is compatible with the new product.

We have:
\[(x \cdot_\chi y)^*=\sum \overline{\chi(x_1,y_1)}y_2^* x_2^* \overline{\chi^{-1}(x_3,y_3)}.\]

Also
\[y^* \cdot_\chi x^*=\sum {\chi(y_1^*,x_1^*)}y_2^* x_2^* {\chi^{-1}(y_3^*,x_3^*)}.\]

Hence, a reasonable hypothesis to work with is to ask $\chi$ to be a $*$--cococycle, compare with (\cite{kn:Majid}).

\begin{defi} We say that the cocycle $\chi$ is a $*$--cocycle provided that the condition
\begin{equation}\label{equation:twistandstar}
\chi(x^*,y^*)= \overline{\chi(y,x)},
\end{equation}
is satisfied.
\end{defi}
\begin{obse} It is clear that in case that $\chi$ is an invertible $*$--cocycle, then $\chi^{-1}$ is also a $*$--cocycle, and then $\mathcal H_\chi$ is also a $*$--Hopf algebra with the same involution that $\mathcal H$ (compare with \cite{kn:Majid}).
\end{obse}

\subsection{Twisting compact quantum groups}

To study the preservation of compactness by Drinfel'd twists we need to consider the positivity of the integral.

If $\varphi$ is a normal left integral in $\mathcal H$, as the cocycle twist does not affect the coproduct, 
$\varphi$ is also a normal integral for $\mathcal H_\chi$.

Hence in this case, assuming that the cocycle satisfies the additional property (\ref{equation:twistandstar}), and considering the two existing products on $H$ one can define two hermitian forms in $H$ as follows:
$\langle x,y \rangle = \varphi(y^* x)$ and $[x,y] = \varphi(y^*\cdot_\chi x)$.

To perform our computations we need the following two formul\ae\/  involving $*$ and $\langle \,\, ,\,\, \rangle$ (see for example \cite{kn:andres-walter-mariana} or section \ref{cqg}).

\begin{equation}\label{equation:1}
\sum \langle x_1,y \rangle \mathcal S(x_2)= \sum \langle x,y_1\rangle y_2^*
\hspace*{.25cm} , \hspace*{.25cm} \sum {\mathcal S}^{-1}(x_1)\langle x_2,y \rangle = \sum y_1^*\langle x,y_2\rangle.
\end{equation}

One has that:

\begin{equation*}[x,y]=\varphi(y^* \cdot_\chi x)= \sum \chi(y_1^*,x_1) \varphi(y_2^* x_2)\chi^{-1}(y_3^*,x_3).
\end{equation*}
Also
\begin{equation*}
\sum \chi(y_1^*,x_1) \langle x_2,y_2 \rangle \chi^{-1}(y_3^*,x_3)
=\sum \chi(y_1^*,x_1) \langle x_2,y_2 \rangle \chi^{-1}(\mathcal S x_3,x_4)
=\sum \chi({\mathcal S}^{-1}x_2,x_1) \langle x_3,y \rangle \chi^{-1}(\mathcal S x_4,x_5),
\end{equation*}
where for the last two equalities above we have used 
equations (\ref{equation:1}).

Similarly
\begin{equation*}
\sum \chi(y_1^*,x_1) \langle x_2,y_2 \rangle \chi^{-1}(y_3^*,x_3)
=\sum \chi(y_1^*, \mathcal S y_2^*) \langle x, y_3 \rangle \chi^{-1}(y_5^*, {\mathcal S}^{-1} y_4^*)
=\overline{\chi}({\mathcal S}^{-1}y_2,y_1)\langle x,y_3\rangle\overline{\chi^{-1}}(\mathcal Sy_4,y_5).
\end{equation*}

Consider the map $\Phi:H \rightarrow H$ defined as: $\Phi(x)=\sum \chi({\mathcal S}^{-1}x_2,x_1) x_3 \chi^{-1}(\mathcal S x_4,x_5)$.

We have shown that:
\begin{equation}\label{equation:positive}
[x, y] = \langle \Phi(x), y \rangle = \langle x, \Phi(y) \rangle.
\end{equation}
Concerning the map $\Phi$ one has the following.

\begin{obse}\label{obse:Phiantipode}
\begin{enumerate}
\item
As $\kappa({\mathcal S}^{-1}x)= \sum \chi({\mathcal S}^{-1}(x_2),x_1)$, the map $\Phi$ can be written as: $\Phi = (\kappa \circ {\mathcal S}^{-1}) \star \operatorname {id} \star \kappa^{-1}$.
Recalling that $\mathcal S_\chi = \kappa \star \mathcal S \star \kappa^{-1}$, we conclude that: $\Phi \star \mathcal S_\chi = (\kappa \circ {\mathcal S}^{-1}) \star \kappa^{-1}$.
Hence, if we call $\zeta= (\kappa \circ {\mathcal S}^{-1}) \star \kappa^{-1}$ we conclude that $\zeta$ is invertible and that $\Phi$ is the convolution inverse of $\mathcal S_\chi \star \zeta^{-1}$.
\item
Moreover, from the definition of $\mathcal S_\chi$ we deduce that: ${\mathcal S}^{-1} \circ \mathcal S_\chi = \kappa \star \operatorname{id} \star \kappa^{-1}$ and comparing with the definition of $\Phi$, we conclude that ${\mathcal S}^{-1}\circ \mathcal S_\chi = \zeta^{-1} \star \Phi$, or equivalently $\Phi = \zeta \star ({\mathcal S}^{-1} \circ \mathcal S_\chi)$.
\item In the particular case that $\zeta = \varepsilon$, i.e. if $\kappa \circ \mathcal S = \kappa$, we deduce that $\Phi = {\mathcal S}^{-1} \circ \mathcal S_\chi$ is the convolution inverse of $\mathcal S_\chi$.

\end{enumerate}
\end{obse}
The next theorem follows directly from the observation above (Observation \ref{obse:Phiantipode}) and equation \eqref{equation:positive}.
\begin{theo}\label{theo:main}
In the hypothesis above, the map $\Phi: H \rightarrow H$ is self adjoint with respect to the original inner product $\langle \,\, ,\,\, \rangle$ in $H$. Moreover, $\mathcal H_\chi$ is a compact quantum group if and only if $\Phi$ is a positive operator, where $\Phi = (\kappa \circ {\mathcal S}^{-1}) \star \operatorname {id} \star \kappa^{-1}$.  
\end{theo}

\section{Matched pairs of CQG}\label{section:matchedpairs}
In this section, we consider the situation of a matched pair of Hopf algebras, as defined in \cite{kn:Majid}, and the
product  naturally obtained in this context --that is called the bicrossproduct--. We study the particular case where
the Hopf algebras in question are equipped with a compatible
$*$-involution and present a necessary and sufficient compatibility condition (linking the involution with the actions)
for the bicrossproduct to be a $*$-Hopf algebra. We prove that if the factors are CQG,
then so is the product.\\
\ \\
A \emph{matched pair} of bialgebras is a quadruple $(A,H,\tl,\tr)$ where $A$ and $H$ are
bialgebras and
\[
\begin{array}{c}
H \stackrel{\triangleleft}{\leftarrow} H\otimes A\stackrel{\triangleright}{\to}A,\\
x\triangleleft a \leftarrow x\otimes a \to x\triangleright a
\end{array}
\]
are linear maps subject to the following compatibility conditions:
\begin{enumerate}
\item
$(A,\tr)$ is a left $H$-module coalgebra, \textit{i.e.} $(A,\tr)$ is a left $H$-module and
\begin{equation}
\Delta_A(x\tr a) = \sum(x_1\tr a_1)\otimes(x_2\tr a_2),\quad \varepsilon_A(x\tr a)=\varepsilon_H(x)\varepsilon_A(a),\quad
\forall x\in H,\ a\in A.
\end{equation}
\item
$(H,\tl)$ is a right $A$-module coalgebra, \textit{i.e.} $(H,\tl)$ is a right $A$-module and
\begin{equation}
\Delta_H(x\tl a) = \sum(x_1\tl a_1)\otimes(x_2\tl a_2),\quad \varepsilon_H(x\tl a)=\varepsilon_H(x)\varepsilon_A(a),\quad
\forall x\in H,\ a\in A.
\end{equation}
\item
The actions $\tr$ and $\tl$ verify the compatibility relations:
\begin{align}
x\tr ab &= \sum(x_1\tr a_1)\big((x_2\tl a_2)\tr b\big),\quad x\tr 1=\varepsilon_H(x)1, 	\label{compat-bi-lr-1}\\
xy\tl a &= \sum\big(x\tl (y_1\tr a_1)\big)(y_2\tl a_2),\quad 1\tl a=\varepsilon_A(a)1, 	\label{compat-bi-lr-2}\\
\sum (x_1\tl a_1)\otimes(x_2\tr a_2) &= \sum (x_2\tl a_2)\otimes(x_1\tr a_1), \label{compat-bi-lr-3}
\end{align}
for all $x,y\in H$ and $a,b\in A$.
\end{enumerate}
\begin{obse}

\begin{enumerate}
\item 
It is well known --see \cite{kn:Majid}, Theorem 7.2.2-- that in the above situation the operations: 
$(a \otimes x)(b \otimes y)= \sum a(x_1 \tr b_1)\otimes (x_2 \tl b_2)y$, $\Delta(a \otimes x)= \sum a_1 \otimes x_1 \otimes a_2 \otimes x_2$ and 
the usual unit and counit, endow the vector space $A \otimes H$ with a bialgebra structure. The bialgebra thus obtained is called the \emph{bicrossproduct} of $A$ and $H$ and will be denoted by $A\bowtie H$.\\

Representing  the element $a \otimes x \in A\bowtie H$ as $ax$, we can abbreviate:

\begin{equation}\label{eq:negra1}
xa=\sum (x_1\tr a_1)(x_2\tl a_2),\quad \Delta(ax)= \sum a_1x_1 \otimes a_2x_2, 
\forall x\in H,\ a\in A.
\end{equation}
Moreover, if $A$ and $H$ are Hopf algebras, then $A\bowtie H$ is also a Hopf algebra with antipode
$$
\ant(ax)=\ant(x)\ant(a)=\sum \big(\ant(x_2)\tr \ant(a_2)\big) \big(\ant(x_1)\tl \ant(a_1)\big) ,\quad
\forall x\in H,\ a\in A.
$$
\end{enumerate}
\end{obse}
\begin{lema} \label{lema:integral}
Let $(A,H,\tl,\tr)$ be a matched pair of cosemisimple Hopf algebras and call $\varphi_A\in A^\vee$ and $\varphi_H\in H^\vee$ its corresponding normal integrals.
Then, $A\bowtie H$ is a cosemisimple Hopf algebra and the linear map $\varphi_{A\bowtie H} : A\bowtie H \rightarrow \C$ defined by
$\varphi_{A\bowtie H}(ax)=\varphi_A(a)\varphi_H(x)$ is a normal integral.  Moreover, the following equations hold:
\begin{equation} \label{eq:pi-AH}
\varphi_A(a)\,\varphi_H(xy)=\sum \varphi_A(x_1\tr a_1)\, \varphi_H\big((x_2\tl a_2)y\big),
\quad
\forall x,y\in H,\ a\in A.
\end{equation}
\begin{equation} \label{eq:2pi-AH}
\varphi_A(ba)\,\varphi_H(x)=\sum \varphi_A(b(x_1\tr a_1))\, \varphi_H\big((x_2\tl a_2)\big),
\quad
\forall x\in H,\ a,b\in A.
\end{equation}
\end{lema}
\dem
The proof that $\varphi_{A\bowtie H}\in (A\bowtie H)^\vee$ is a normal integral on $A\bowtie H$ follows directly from the definitions; hence $A\bowtie H$ is cosemisimple.

Let us call $\mathcal N: A\bowtie H \rightarrow A\bowtie H$ the Nakayama automorphism of $A\bowtie H$.
Observe that from equation \eqref{eq:nakayama} we deduce that $A$ and $H$ are $\mathcal N$-invariant, and therefore
it follows by the definition of $\mathcal N$ --recall \eqref{eq:naka}-- that the restrictions $\mathcal N|_A$ and $\mathcal N|_H$ are the Nakayama automorphisms of $A$ and $H$, respectively.

To prove \eqref{eq:pi-AH} we take $a\in A$ and $x,y\in H$ and compute:
\begin{align*}
\varphi_A(a)\,\varphi_H(xy)
=
\varphi_A(a)\,\varphi_H(y\mathcal N(x))
&=
\varphi_{A\bowtie H}(ay\mathcal N(x))
=
\varphi_{A\bowtie H}(xay)\\
=
\varphi_{A\bowtie H}\left(\sum(x_1\tr a_1)(x_2\tl a_2)y\right)
&=
\sum \varphi_A(x_1\tr a_1)\,\varphi_{H}\big((x_2\tl a_2)y\big).
\end{align*}

Equation \eqref{eq:2pi-AH} is proved similarly using $\mathcal N^{-1}$ instead of $\mathcal N$ and formula \eqref{eq:-nakayama}.
\fin

\begin{teo} \label{teo:bicrossproduct-CQG}
Let $(A,H,\tl,\tr)$ be a matched pair of Hopf algebras where $A$ and $H$ are $*$-Hopf algebras.
Then there exists a structure of $*$-Hopf algebra in $A\bowtie H$ such that $A$ and $H$ are sub-$*$-Hopf algebras
of $A\bowtie H$ if and only if
\begin{equation} \label{eq:compat-cerito-acciones}
a^*\varepsilon_H(x^*)=\sum (x_2\tl a_2)^*\tr (x_1\tr a_1)^*
\textrm{ and }
\varepsilon_A(a^*)x^*=\sum (x_2\tl a_2)^*\tl (x_1\tr a_1)^*
\quad
\forall x\in H,\ a\in A.
\end{equation}

Moreover, in this situation $A$ and $H$ are CQG if and only if  $A\bowtie H$ is a CQG.
\end{teo}
\dem
As any  $*$-structure in $A\bowtie H$ as above will satisfy:
\begin{equation} \label{eq:cero-bismash}
(ax)^*=x^*a^*=\sum (x_1^*\tr a_1^*)(x_2^*\tl a_2^*),\quad
\forall x\in H,\ a\in A.
\end{equation}
Hence it is natural to take the above equation \eqref{eq:cero-bismash} as the definition of a $*$-structure on $A\bowtie H$. Clearly this $*$--structure  is comultiplicative on $A\bowtie H$ and restricts to $A$ and $H$ as stated.
To prove the antimultiplicativity we observe that:$\big((ax)(by)\big)^*=
\left(\sum a (x_1\tr b_1)(x_2\tl b_2) y \right)^*
=\sum y^* (x_2\tl b_2)^* (x_1\tr b_1)^* a^* ,
(by)^*(ax)^*=y^*b^*x^*a^*.$
Then $*$ is antimultiplicative if and only if
\begin{align*}
a^*x^*=
\sum (x_2\tl a_2)^* (x_1\tr a_1)^*=
\sum \big((x_2\tl a_2)^*_1\tr (x_1\tr a_1)^*_1\big) \big((x_2\tl a_2)^*_2\tl (x_1\tr a_1)^*_2\big),[\textrm{using \eqref{eq:negra1}}].
\end{align*}
Or, in the tensor product notation,  if and only if
\begin{align*}
a^*\otimes x^* &=
\sum \big((x_3\tl a_3)^*\tr (x_1\tr a_1)^*\big) \otimes \big((x_4\tl a_4)^*\tl (x_2\tr a_2)^*\big)\in A\otimes H
, \quad \forall a\in A,\ x\in H.
\end{align*}
Then, the relations \eqref{eq:compat-cerito-acciones} follows by applying $\operatorname{id}_A \otimes \varepsilon_H$ and $\varepsilon_A \otimes \operatorname{id}_H$ to the equation above. 

Conversely, if we assume equations \eqref{eq:compat-cerito-acciones}, then:
\begin{align*}
\sum(x_2\tl a_2)^* (x_1\tr a_1)^*
&=
\sum \big((x_2\tl a_2)^*_1\tr (x_1\tr a_1)^*_1\big) \big((x_2\tl a_2)^*_2\tl (x_1\tr a_1)^*_2\big)
&=
\sum \big((x_3\tl a_3)^*\tr (x_1\tr a_1)^*\big) \big((x_4\tl a_4)^*\tl (x_2\tr a_2)^*\big)
\\
&=
\sum \big((x_2\tl a_2)^*\tr (x_1\tr a_1)^*\big) \big((x_4\tl a_4)^*\tl (x_3\tr a_3)^*\big)
&& [\textrm{using \eqref{compat-bi-lr-3}}]
\\
&=
\sum \big (a^*_1\varepsilon_H(x^*_1))( x^*_2\varepsilon_A(a^*_2)) \ = \ a^*x^*.
\end{align*}
We deduce that in this situation  $*$ is antimultiplicative, and then, the fact that it is an involution follows immediately.

\medskip
We assume now that $A$ and $H$ are CQG and verify conditions \eqref{eq:compat-cerito-acciones}.
Let $\varphi_A\in A^\vee$ and $\varphi_H\in H^\vee$ be the normal integrals.
As $A$ and $H$ are CQG, the hermitian forms
$\langle\ ,\ \rangle_A$ and $\langle\ ,\ \rangle_H$ defined by the
formul\ae:
$\langle a,b\rangle_A=\varphi_A(b^*a), 
\langle x,y\rangle_H=\varphi_H(y^*x),
\forall a,b\in H,\ x,y\in H$,
are inner products.

\medskip
Considering the previously defined normal integral $\varphi_{A\bowtie H}: A\bowtie H \rightarrow \C$ and computing explicitly its associated hermitian form one obtains that:
\begin{equation}\label{eq:rels-pi}
\begin{split}
\langle ax,by \rangle_{A\bowtie H}&=
\varphi_{A\bowtie H}\big((by)^*ax\big)
=\varphi_{A\bowtie H}\big(y^*b^*ax\big)
=\sum \varphi_A \big(y^*_1\tr (b^*a)_1\big) \varphi_H \Big(\left (y^*_2\tl (b^*a)_2 \right )x\Big)\\
&= \varphi_A(b^*a)\varphi_H(y^*x)= \langle a,b\rangle_A \langle x,y\rangle_H.
\end{split}
\end{equation}

The penultimate equation above, follows directly from condition \eqref{eq:pi-AH}.
It is then clear that when $A$ and $H$ are CQGs, the hermitian form:
$\langle ax,by \rangle_{A\bowtie H}=\langle a,b\rangle_A \langle x,y\rangle_H,\quad
\forall a,b\in H,\ x,y\in H$, is indeed an inner product.
\qed

\subsection{The quantum double}
The Drinfel'd double of a finite Hopf algebra is a particular case of the above construction, and thus we can apply to this situation the results just proved. In that sense, using condition (\ref{eq:compat-cerito-acciones}) we deduce that $H$ is a CQG if and only if, its double is a CQG.\\
\ \\
Let $H$ be a finite Hopf algebra. We consider the actions
$H \stackrel{\triangleleft}{\leftarrow} H\otimes H^\vee\stackrel{\triangleright}{\to}H^\vee$ defined by
\begin{equation*} 
(x\tr\alpha)(y)=\sum \alpha\left(\ant^{-1}(x_2) y x_1\right),\quad
x\tl \alpha = \sum \alpha\left(\ant^{-1}(x_3)x_1 \right) x_2
,\quad \forall \alpha\in H^\vee,\ x,y\in H.
\end{equation*}
In this situation it is well known that $\left(H^{\vee \cop},H,\tl,\tr\right)$ is a matched pair of Hopf algebras and
the bicrossproduct $D(H)=H^{\vee \cop}\bowtie H$ is called the {\em quantum double} --or the {\em Drinfel'd double}-- of $H$.
Hence $D(H)$ is a Hopf algebra with the property that $H^{\vee \cop}$ and $H$ are Hopf subalgebras and $D(H)=H^{\vee \cop}H$.
In explicit terms the Hopf structure maps of $D(H)$ are:
\begin{align*}
(\alpha x)(\beta  y) &=
\sum \alpha (x_1\tr \beta_2)(x_2\tl \beta_1) y=
\sum \alpha \left(x_1\rightharpoonup \beta \leftharpoonup \ant^{-1}(x_3)\right) x_2 y,\\
\Delta(\alpha x) &= \sum \alpha_2 x_1 \otimes \alpha_1 x_2,\\
\ant(\alpha x) &=
\sum \left(\ant(x_2)\tr \ant(\alpha_1)\right) \left(\ant(x_1)\tl \ant(\alpha_2)\right)
=\sum \big(\ant(x_3)\rightharpoonup \ant^{-1}(\alpha) \leftharpoonup x_1 \big) \ant(x_2),
\end{align*}
for all $\alpha,\beta\in H^\vee$ and $x,y\in H$.

\begin{coro} \label{coro:double-cqg}
A finite Hopf algebra $H$ is a CQG if and only if its quantum double
$D(H)$ is a CQG.
\end{coro}
\dem
We assume first that $H$ is a $*$-Hopf algebra and we endow $H^\vee$ with the $*$-structure given in \eqref{eq:*-dual}.
We show that the first of the equations \eqref{eq:compat-cerito-acciones} is verified --the other is similar--:

\begin{align*}
\sum(x_2\tl \alpha_1)^*\tl (x_1\tr \alpha_2)^*
&=\sum(x_1\tr \alpha_2)^* \Big(\ant^{-1}\big((x_2\tl \alpha_1)^*_3\big)\,(x_2\tl \alpha_1)^*_1 \Big)(x_2\tl \alpha_1)^*_2 \\
=
\sum \overline{(x_1\tr \alpha_2)\Big((x_2\tl \alpha_1)_3\,\ant^{-1} \big((x_2\tl \alpha_1)_1\big)\Big)}
(x_2\tl &\alpha_1)^*_2 = \Big(\sum (x_1\tr \alpha_2) \Big((x_2\tl \alpha_1)_3\,\ant^{-1} \big((x_2\tl \alpha_1)_1\big)\Big)
(x_2\tl \alpha_1)_2\Big)^*
\\
=
\Big(\sum \alpha_2 \Big(\ant^{-1}(x_2)\, (x_3\tl \alpha_1)_3\,\ant^{-1}\big((x_3\tl \alpha_1)_1\big)\,x_1 \Big)&
(x_3\tl \alpha_1)_2\Big)^*
=
\Big(\sum \alpha_1\big(\ant^{-1}(x_7)x_3\big) \alpha_2 \big(\ant^{-1}(x_2)\, x_6\ant^{-1}(x_4)x_1\big) x_5\Big)^*
\\
=\left(\sum \alpha\big(\ant^{-1}(x_7)x_3 \ant^{-1}(x_2)\, x_6\ant^{-1}(x_4)x_1\big) x_5\right)^*
&=
\overline{\alpha(1)} x^*
,\quad  \forall \alpha\in H^\vee, x\in H.
\end{align*}
Then the result follows directly from Theorem \ref{teo:bicrossproduct-CQG}.
\qed
\begin{obse}
If $H$ is a finite CQG, then the corresponding $*$-structure in $D(H)$ is defined by
\begin{equation*} 
(\alpha x)^*=
\sum \left( x_1^*\rightharpoonup \alpha^*\leftharpoonup \ant^{-1}(x_3)^*\right) x_2^* =
\sum \left( x_3 \rightharpoonup \alpha \leftharpoonup \ant^{-1}(x_1)\right)^* x_2^*,
\end{equation*}
for all $\alpha\in H^{\vee},\ x\in H$.
\end{obse}

\subsection{Yetter-Drinfel'd modules}
It is well known that for a finite Hopf algebra $H$, the Yetter-Drinfel'd modules over $H$ can be interpreted as modules over $D(H)$. In view of the general results just proved, the compactness of $H$ (in the finite case) can be expressed in terms of the category of $D(H)$--modules.
Hence, one produces a characterization of the compactness of a finite $*$-Hopf algebra in terms of its category of Yetter-Drinfel'd modules.

Assume that $H$ is also endowed with a $*$--operation and take in $D(H)$ the involution defined above.  In this situation, not only a vector space can be equivalently viewed as a Yetter--Drinfeld module or as a $D(H)$--module but it is also equivalent for an hermitian form to be invariant for the module and comodule structures than to be invariant for the $D(H)$--structure. Recall the definition of inner product invariant for an action that appears in Observation \ref{rem:*representation}.

\begin{prop}
Let $H$ be a finite $*$--Hopf algebra and take $V\in {}^H_H{\mathcal {YD}}$. Assume that $\langle\,\, ,\, \rangle$ is an hermitian form in $V$. Then, $\langle\,\, ,\, \rangle$ is
invariant for the $D(H)$--module structure if and only if it is simultaneously invariant for the $H$--module and for the $H$--comodule structures.
\end{prop}
\dem
Assume first that $\langle\,\, ,\, \rangle$ is invariant for the $D(H)$-module structure:
\[
\langle \alpha x\cdot u ,v \rangle = \langle u ,(\alpha x)^*\cdot v \rangle,
\quad \forall \alpha\in H^\vee,\ x\in H,\ u,v\in V.
\]
Being $H$ and $H^\vee$ $*$-subalgebras of $D(H)$, we deduce that $\langle\,\, ,\, \rangle$ is invariant for the $H$--module and for the $H^\vee$-module structures.
The fact that the  $H^\vee$-- invariance of $\langle\,\, ,\, \rangle$
is equivalent to the invariance for the corresponding $H$-comodule
structure is the content of Observation \ref{rem:*representation}.

Conversely, assume that $\langle\,\, ,\, \rangle$ is
invariant for the $H$--module and for the $H$--comodule structures. Arguing as
above we deduce that
$\langle\,\, ,\, \rangle$ is invariant for the $H^\vee$--module structure.
As $D(H)=H^\vee H$ and $H\subset D(H)$, $H^\vee\subset D(H)$ as $*$-subalgebras, we conclude that
$\langle\,\, ,\, \rangle$ is invariant for the $D(H)$--module structure.
\qed

\section{Extensions of CQG}\label{section:extensions}

We consider in this section the case of extensions of Hopf algebras and more particularly, the case of cleft and cocleft
extensions. It is proved in \cite{kn:andrus-devoto}, that cleft and cocleft extensions are exactly the Hopf algebras
arising from what we call here a cocycle linked pair of Hopf algebras (by duality this concept is closely related to
the notion of matched pair of Hopf algebras --see Observation \ref{rem:bi-algcoalg}).
We consider the case in which the Hopf algebras involved are CQG, and --inspired on the results of \cite{kn:andrus2}--give necessary and sufficient conditions for the product to be a $*$-Hopf algebra and a CQG.

\smallskip
A \emph{cocycle linked pair of bialgebras} (\emph{of Hopf algebras})
is a sextuple $(A,H,\trb,\rho,\chi,\psi)$, where $A$ and $H$ are bialgebras (Hopf algebras) and
\begin{equation} \label{def-clpb}
\begin{array}{c}
H\otimes A\stackrel{\trb}{\to}A\\
x\otimes a \mapsto x\trb a
\end{array}
,\quad
\begin{array}{c}
 H\stackrel{\rho}{\to}H\otimes A\\
 x\mapsto \sum x_H\otimes x_A
\end{array},
\begin{array}{c}
H\otimes H\stackrel{\chi}{\to}A\\
x\otimes y \mapsto \chi(x,y)
\end{array}
,\qquad
\begin{array}{c}
 H\stackrel{\psi}{\to}A\otimes A\\
 x\mapsto \sum x_I\otimes x_{II}
\end{array},
\end{equation}
are linear maps such that: (compare with \cite{kn:andrus-devoto}, \cite{kn:andrus2} and \cite{kn:Majid})
\begin{enumerate}
\item
$(A,\trb,\chi)$ is a \emph{cocycle left $H$-module algebra}:
\begin{align}
x\trb 1  &= \varepsilon(x)1, \label{a1}  \\
x\trb ab &= \sum (x_1\trb a)(x_2\trb b), \label{a2} \\
1\trb a &= a, \label{a3} \\
\sum \big( x_1\trb (y_1\trb a)\big)\,\chi(x_2,y_2)
&=
\sum \chi(x_1,y_1)\,\big( (x_2 y_2)\trb a\big), \label{a4} \\
\sum \big( x_1\trb \chi(y_1,z_1)\big)\,\chi(x_2,y_2z_2)
&=
\sum \chi(x_1,y_1)\, \chi(x_2 y_2,z), \label{a5} \\
\chi(x,1)&= \chi(1,x)=\varepsilon(x)1, \label{a6}
\end{align}
for all $x,y,z\in H$, $a,b\in A$.
\item
$(H,\rho,\psi)$ is a \emph{cocycle right $A$-comodule coalgebra}:
\begin{align}
\sum\varepsilon(x_H)x_A
&=
\varepsilon(x)1 , \label{b1} \\
\sum x_{H1} \otimes x_{H2} \otimes x_A
&=
\sum x_{1H} \otimes x_{2H} \otimes x_{1A}x_{2A}, \label{b2} \\
\sum x_H\varepsilon(x_A)
&=
x, \label{b3} \\
\sum x_{2HH} \otimes x_{1I} x_{2HA} \otimes x_{1II} x_{2A}
&=
\sum x_{1H} \otimes x_{1A1} x_{2I} \otimes x_{1A2} x_{2II}, \label{b4} \\
\sum x_{1I1} x_{2HI} \otimes x_{1I2} x_{2HII}
\otimes x_{1II} x_{2A}
&=
\sum x_{1I} \otimes x_{1II1} x_{2I} \otimes  x_{1II2}x_{2II},  \label{b5} \\
\sum\varepsilon(x_I)x_{II}
&=
\sum x_I\varepsilon(x_{II})=\varepsilon(x)1, \label{b6}
\end{align}
for all $x\in H$.

\item
the maps $\trb,\ \rho,\ \chi,\ \psi$ verify the following compatibility relations:
\begin{align}
\varepsilon(x\trb a) 		
&=
\varepsilon(x)\varepsilon(a), \label{c1}\\
\sum 1_A\otimes 1_H
&=
1\otimes 1 ,\label{c2}\\
\varepsilon\big(\chi(x,y)\big) 		
&=
\varepsilon(x)\varepsilon(y), \label{c3}\\
\sum 1_I\otimes 1_{II}
&=
1\otimes 1 ,\label{c4}\\
\sum (x_1\trb a)_1 x_{2I} \otimes (x_1\trb a)_2 x_{2II}&=
\sum x_{1I} (x_{2H}\trb a_1) \otimes x_{1II} x_{2A} (x_3\trb a_2), \label{c5}\\
\sum (x_2y_2)_H\otimes \chi(x_1,y_1) (x_2y_2)_A	
&=
\sum x_{1H}\, y_{1H} \otimes x_{1A} (x_2\trb y_{1A})\chi(x_3,y_2), \label{c6}\\
\sum x_{2H} \otimes (x_1\trb a)x_{2A}
&=
\sum x_{1H} \otimes x_{1A} (x_2\trb a), \label{c7} \\
\sum \chi(x_1,y_1)_1 (x_2y_2)_I \otimes \chi(x_1,y_1)_2 (x_2y_2)_{II}
&=
\sum x_{1I} (x_{2H}\trb y_{1I}) \chi(x_{4H}, y_{2H}) \otimes \notag \\
& \ \ \ \
x_{1II} x_{2A} (x_3\trb y_{1II}) x_{4A} (x_5\trb y_{2A}) \chi(x_6,y_3),\label{c8}
\end{align}
for all $x,y\in H$, $a\in A$.
\end{enumerate}

\begin{obse}\label{rem:action}
\begin{enumerate}
\item \label{rem:action-2} In this context , see for example \cite{kn:andrus-devoto}, \cite{kn:andrus2} and
\cite{kn:Majid} Theorem 6.3.9, one can define a bialgebra structure on $A \otimes H$, that is called the \emph{cocycle bismash product} of $A$ and $H$ and it is usually denoted by $A^\psi \#_\chi H$. The basic operations --computed on the elementary tensors of $A^\psi \#_\chi H$, that will be denoted generically as $a \# x$ with $a \in A, x \in H$-- are the following: 

\begin{itemize}
\item[{\em Product:}] $(a\# x)(b\# y) = \sum a(x_1\trb b) \chi(x_2,y_1)\# x_3 y_2$, with unit $1\#1$;
\item[{\em Coproduct:}] $\Delta(a\# x) = \sum a_1 (x_1)_I\# {(x_2)}_H \otimes a_2 (x_1)_{II} (x_2)_A\# x_3$,
with counit $\varepsilon \otimes \varepsilon$.
\end{itemize}
\item Let $(A,H,\trb,\rho,\chi,\psi)$ be a cocycle linked pair of bialgebras.
Condition \eqref{a4} shows that $\trb$ is not necessarily an action.
If $\chi$ is invertible and belongs to the center of the convolution algebra $\Hom(H\otimes H,A)$, then
$\chi$ cancels in \eqref{a4} and $\trb$ becomes an action. This is the case for example
 whenever $\chi$ is \emph{trivial}, i.e. $\chi = \varepsilon \otimes \varepsilon$  or whenever $A$ is commutative,
$H$ is cocommutative and $\chi$ is convolution invertible.
Dual considerations hold  for $\rho$ and $\psi$. This particular situation can be generalized to the notion of {\em cocycle Singer pair} introduced in Definition \ref{defi:csingerpair}. 

\item For the product of two elements of the form
$1 \# x, 1 \# y \in A^\psi \#_\chi H$, we have the following formula: $(1 \# x)(1 \# y)=\sum \chi(x_1,y_1) \# x_2y_2$.

\item The rather complicated compatibility condition \eqref{c8} can be samewhat simplified as we can write it as: 
\begin{equation}\label{eqn:19'}
\Delta \chi \star \psi m = (\psi \otimes \varepsilon) \star \theta \star (1 \otimes \chi),
\end{equation} 
with 
\begin{equation*}\theta(x,y)= \sum (x_{1H} \trb y_{1I})\chi(x_{3H},y_{2H})\otimes x_{1A}(x_2 \trb y_{1II})x_{3A}(x_4 \trb y_{2A})
\end{equation*}
\begin{equation}=\sum (x_{1H} \trb y_{1I})\chi(x_{2H},y_{2H})\otimes x_{1A}x_{2A}(x_3 \trb y_{1II})(x_4 \trb y_{2A}),
\end{equation}
where the last equality follows from a direct application of equation \eqref{c7}.
Using equations \eqref{a2} and \eqref{b2} the above expression can be transformed into: 
\begin{equation}\theta(x,y)= \sum (x_{1H1} \trb y_{1I})\chi(x_{1H2},y_{2H})\otimes x_{1A}(x_2 \trb (y_{1II}y_{2A}))\label{eqn:imp}.
\end{equation}
\end{enumerate}

\end{obse}
\medskip

Let
\(
(M):\ A\stackrel{\iota}{\to}M\stackrel{\pi}{\to}H
\)
be a sequence of Hopf algebras and Hopf algebra maps.
In this situation the correstriction
$m\mapsto \sum m_1\otimes \pi(m_2): M \rightarrow M \otimes H$ endows $M$ with a structure of right $H$-comodule algebra
and the restriction $a\otimes m\mapsto \iota(a)m: A \otimes M \rightarrow M$ endows $M$ with a structure of a left
$A$-module coalgebra.
The space of right coinvariants of $M$ is $M^{\rm{co}H}=\{m\in M:\ \sum m_1\otimes\pi(m_2)=m\otimes 1\}$.

\begin{defi}[\cite{kn:andrus-devoto}] \label{def:exacta}
We say that $(M)$ is \emph{exact} and that $M$ is an \emph{extension} of $H$ by $A$ if
\begin{enumerate}
\item
$\iota$ is injective,
\item
$\pi$ is surjective,
\item
$\Im(\iota)=M^{\rm{co}H}$,
\item
$\Ker(\pi)=MA^+$, being $A^+$ the kernel of the counit of $A$.
\end{enumerate}

An \emph{equivalence} of exact sequences $(M)$ and $(M')$ is a Hopf algebra map $M\to M'$ which induces the identity map
in $A$ and $H$. This algebra map is necessarily bijective.

The extension is \emph{cleft} if there exists a convolution invertible $H$-comodule map $H\to M$, and is
\emph{cocleft} if there exists a convolution invertible $A$-module map $M\to A$.
\end{defi}

\begin{prop} [Lemma 3.2.17, \cite{kn:andrus-devoto}] \label{teo:ext-hopf}
Let $(A,H,\trb,\rho,\chi,\psi)$ be a cocycle linked pair of Hopf algebras,
where $\chi:H\otimes H\to A$ and $\psi:H\to A\otimes A$ are convolution invertible.
Then $A^\psi\#_\chi H$ is a Hopf algebra and
\[
A\stackrel{\iota_A}{\longrightarrow}A^\psi\#_\chi H\stackrel{\pi_H}{\longrightarrow}H
\]
is a cleft and cocleft extension of Hopf algebras.
Conversely, any cleft and cocleft extension of Hopf algebras $A\to M\to H$ is equivalent to $A^\psi\#_\chi H$,
for some $\trb,\rho,\chi,\psi$ as above, where $\chi$ and $\psi$ are convolution invertible.

In this situation, if we write the inverses of $\chi$ and $\psi$ as $x\otimes y \mapsto \chi^{-1}(x,y)$
and $x\mapsto \sum x_{\widehat I}\otimes x_{\widehat{II}}$ respectively,

then the antipode of $A^\psi\#_\chi H$ is given by
\begin{align*}
\ant(a\# x)&=\sum \Big(\chi^{-1}\big(\ant (x_{2H}), x_{3H}\big) \# \ant(x_{1H}) \Big)
\Big(x_{4\widehat I}\,\ant\big(ax_{1A}x_{2A}x_{3A}
x_{4\widehat{II}}\big)\# 1\Big)\\
&=\sum \Big(\chi^{-1}\big(\ant (x_{1H2}), x_{1H3}\big) \# \ant(x_{1H1}) \Big)
\Big(x_{2\widehat I}\,\ant\big(ax_{1A}
x_{2\widehat{II}}\big)\# 1\Big)
\end{align*}
\fin
\end{prop}

It is well known (see Theorem 3.3 in \cite{kn:schneider}) that an exact sequence of finite Hopf algebras is always cleft and cocleft, and --being $\iota$ injective
and $\pi$ surjective-- conditions (3) and (4) in Definition \ref{def:exacta} are equivalent (\cite{kn:masuoka}).

\begin{obse} \label{rem:ext}
In the context considered above, the bialgebra $A^\psi\#_\chi H$ can be endowed with natural structures of left $A$-module
and a right $H$-comodule as follows:
\[
a'\cdot(a\# x)=a'a\# x,\qquad \delta(a\# x)=\sum a\# x_1\otimes x_2,\quad \forall a\in A,\ x\in H.
\]
Observe that $a'\cdot(a\# x)= (a' \# 1)(a\# x)$, and
$\delta= \big((\id_A \otimes \id_H) \otimes (\varepsilon \otimes \id_H)\big)\Delta$.

If we consider the natural inclusion and projection maps
\[
\begin{array}{c}
A\stackrel{\iota_{_A}}{\longrightarrow} A^\psi\#_\chi H\\
a\mapsto a\# 1
\end{array}
,\quad
\begin{array}{c}
H\stackrel{\iota_{_H}}{\longrightarrow} A^\psi\#_\chi H\\
x\mapsto 1\# x
\end{array}
,\quad
\begin{array}{c}
A^\psi\#_\chi H\stackrel{\pi_{_A}}{\longrightarrow} A\\
a\# x\mapsto a\varepsilon(x)
\end{array}
,\quad
\begin{array}{c}
A^\psi\#_\chi H\stackrel{\pi_{_H}}{\longrightarrow} H\\
a\# x\mapsto \varepsilon(a)x
\end{array}
,
\]
then, $\iota_A$ and $\pi_H$ are bialgebra maps, $\iota_H$ is a morphism of $H$-comodules and $\pi_A$ is
morphism of $A$-modules.
\end{obse}

\begin{lema} \label{lema:cocycle-bismash-coss}
Let $(A,H,\trb,\rho,\chi,\psi)$ be a cocycle linked pair of cosemisimple bialgebras and call
$\varphi_A:A^\vee \rightarrow \C$ and $\varphi_H: H^\vee \rightarrow \C$ its corresponding normal integrals.
Then, the linear map $\Phi : A\# H \rightarrow \C$ defined by $\Phi(a\# x)=\varphi_A(a)\varphi_H(x)$ is a normal integral on
$A^\psi\#_\chi H$ that is then cosemisimple.
\end{lema}
\dem
The following computation shows that $\Phi$ is a left integral:
\begin{align*}
\sum (a\# x)_1 \Phi\big((a\# x)_2\big)
&=
\sum a_1 (x_1)_I\# {(x_2)}_H \Phi\big( a_2 (x_1)_{II} (x_2)_A\# x_3\big)\\
&=
\sum a_1 (x_1)_I\# {(x_2)}_H \varphi_A\big( a_2 (x_1)_{II} (x_2)_A\big) \varphi_H( x_3) \\
&=
\sum \varphi_H(x) a_1 1_I\# 1_H \varphi_A\big( a_2 1_{II} 1_A\big)
=
\sum \varphi_H(x) a_1 \# 1_H \varphi_A(a_2) = \Phi(a\# x)1\# 1.
\end{align*}

As $\Phi(1\# 1)=1$ it follows that $A\# H$ is
cosemisimple and hence that $\Phi$ is also right integral.
\qed

\begin{defi}
An \emph{extension of $*$-Hopf algebras} is an extension
\(
(M):\ A\stackrel{\iota}{\to}M\stackrel{\pi}{\to}H
\)
where $A$, $M$ and $H$ are $*$-Hopf algebras and $\iota$ and $\pi$ are $*$-maps.
\end{defi}

The following theorem is inspired in \cite{kn:andrus2}, Proposition 3.2.9.

\begin{teo}
\label{teo:* en AH}
Let $(A,H,\trb,\rho,\chi,\psi)$ be a cocycle linked pair of bialgebras, where $\chi:H\otimes H\to A$ and $\psi:H\to A\otimes A$
are convolution invertible and $A$ and $H$ are $*$-Hopf algebras.
Let  $\gamma:H\to A$ be a linear map such that $\gamma(1)=1$ and $\varepsilon\gamma=\varepsilon$.
Then the formula below:
\begin{equation} \label{eq:estrella en bismash}
(a\# x)^*
=\sum \gamma\cc x_1^*\dd \cc x_2^*\trb a^*\dd\# x_3^*,\quad \forall  a\in A,\ x\in H
\end{equation}
defines a structure of $*$-Hopf algebra in $A^\psi\#_\chi H$ if and only if
\begin{align}
\sum \gamma\cc x_2^* \dd\cc x_3^*\trb \gamma(x_1)^*\dd
&=
\varepsilon(x^*)1, 			        \label{cero}\\
\sum \gamma\cc x_2^* \dd\cc x_3^*\trb(x_1\trb a)^*\dd
&=
a^*\gamma(x^*), 			        \label{uno}\\
\sum \gamma(y_1^*)\cc y_2^*\trb \gamma(x_1^*)\dd\chi(y_3^*,x_2^*)
&=
\sum \gamma(y_2^*x_2^*) \cc(y_3^*x_3^*)\trb\chi(x_1,y_1)^*\dd,		\label{dos} \\
\sum (x_2^*)_H\otimes\gamma(x_1^*)(x_2^*)_A
&=
\sum \left({(x_1)}_H\right)^*\otimes \gamma(x_2^*)\left(x_3^*\trb\left({(x_1)}_A\right)^*\right), \label{tres} \\
\sum \big(\gamma(x_1^*)\big)_1 (x_2^*)_I\otimes \big(\gamma(x_1^*)\big)_2 (x_2^*)_{II}
&=
\sum \gamma\cc\big((x_2)_H\big)^*\dd \Big(\big((x_3)_H\big)^*\trb \big((x_1)_I\big)^*\Big) \otimes \nonumber \\
&\ \ \ \ \  \ \ \
\gamma(x_4^*) \Big(x_5^*\trb \big((x_1)_{II}(x_2)_A(x_3)_A\big)^*\Big),
\label{cuatro}
\end{align}
for all $a\in A$, $x\in H$. Hence in this situation $A\to A^\psi\#_\chi H\to H$ is an extension of $*$-Hopf algebras.
\end{teo}
\dem
First we assume that conditions \eqref{cero}--\eqref{cuatro} are verified.
The fact that the operator $*$ is an involution, i.e. that  $(a\# x)^{**}=a\# x$, follows easily from \eqref{cero} and \eqref{uno}.

Next we show that $*$ is antimultiplicative:
\begin{align*}
\big((a\# x)(b\# y)\big)^*
&=
\sum\big( a(x_1\trb b)\chi(x_2,y_1)\# x_3 y_2\big)^* \\
&=
\sum  \gamma(y_2^*x_3^*)\Big( y_3^*x_4^*\trb \big(\chi(x_2,y_1)^* (x_1\trb b)^* a^*\big)\Big)\# y_4^*x_5^* \\
&=
\sum \gamma(y_2^*x_3^*) \big( y_3^*x_4^*\trb \chi(x_2,y_1)^* \big) \big( y_4^*x_5^*\trb (x_1\trb b)^*\big)
\left( y_5^*x_6^*\trb a^*\right) \# y_6^*x_7^*
&& [\textrm{using \eqref{a2}}]
\\
&=
\sum  \gamma(y_1^*) \big( y_2^*\trb \gamma(x_2^*) \big) \chi(y_3^*,x_3^*) \big( y_4^*x_4^*\trb (x_1\trb b)^*\big)
\left( y_5^*x_5^*\trb a^*\right) \# y_6^*x_6^*
&& [\textrm{using \eqref{dos}}]
\\
&=
\sum  \gamma(y_1^*) \big( y_2^*\trb \gamma(x_2^*) \big) \Big( y_3^*\trb \big(x_3^*\trb (x_1\trb b)^*\big)\Big) \chi(y_4^*,x_4^*)
\left( y_5^*x_5^*\trb a^*\right) \# y_6^*x_6^*
&& [\textrm{using \eqref{a4}}]
\\
&=
\sum  \gamma(y_1^*) \big( y_2^*\trb \Big( \gamma(x_2^*) \big) \big(x_3^*\trb (x_1\trb b)^*\big)\Big) \chi(y_3^*,x_4^*)
\left( y_4^*x_5^*\trb a^*\right) \# y_5^*x_6^*
&& [\textrm{using \eqref{a2}}]
\\
&=
\sum  \gamma(y_1^*) \big( y_2^*\trb ( b^*\gamma(x_1^*))\big) \chi(y_3^*,x_2^*)
\left( y_4^*x_3^*\trb a^*\right) \# y_5^*x_4^*
&& [\textrm{using \eqref{uno}}]
\\
&=
\sum \gamma(y_1^*) \big(y_2^*\trb (b^*\gamma(x_1^*)\big) \big(y_3^*\trb(x_2^*\trb a^*)\big) \chi(y_4^*,x_3^*) \# y_5^*x_4^*
&& [\textrm{using \eqref{a4}}]
\\
&=
\sum \gamma(y_1^*) \Big(y_2^*\trb \big(b^*\gamma(x_1^*)(x_2^*\trb a^*)\big)\Big) \chi(y_3^*,x_3^*) \# y_4^*x_4^*
&& [\textrm{using \eqref{a2}}]
\\
&=
\sum  \big(\gamma(y_1^*) (y_2^*\trb b^*\# y_3^*)\big)\big(\gamma(x_1^*) (x_2^*\trb a^*\# x_3^*)\big)
\\
&=(b\# y)^{*}(a\# x)^{*}.
\end{align*}
In order to prove the comultiplicativity of $*$, it is enough to show that
\[
\Delta\big((a\# 1)^*\big)=\big(\Delta(a\# 1)\big)^{*\otimes *},\qquad
\Delta\big((1\# x)^*\big)=\big(\Delta(1\# x)\big)^{*\otimes *},\qquad \forall a\in A,\ x\in H.
\]
The first equality is obvious, for the second:
\begin{align*}
\big(\Delta(1\# x)\big)^{*\otimes *}
&=
\sum   \big((x_1)_I\# {(x_2)}_H\big)^* \otimes \big((x_1)_{II} (x_2)_A\# x_3\big)^*
\\
&=
\sum \gamma\Big(\big((x_2)_H\big)_1^*\Big) \Big(\big((x_2)_H\big)_2^*\trb \big((x_1)_I\big)^*\Big) \# \big((x_2)_H\big)_3^* \otimes
\\
& \ \ \ \ \ \ \ \ \
\gamma(x_3^*) \Big( x_4^*\trb \big( (x_1)_{II} (x_2)_A \big)^* \Big) \# x_5^*
\\
&=
\sum \gamma\Big(\big((x_2)_H\big)^*\Big) \Big(\big((x_3)_H\big)^*\trb \big((x_1)_I\big)^*\Big) \# \big((x_4)_H\big)^* \otimes
\\
& \ \ \ \ \ \ \ \ \
\gamma(x_5^*) \Big( x_6^*\trb \big( (x_1)_{II} (x_2)_A(x_3)_A (x_4)_A \big)^* \Big) \# x_7^*
&& [\textrm{using \eqref{b2}}]
\\
&=
\sum \gamma\Big(\big((x_2)_H\big)^*\Big) \Big(\big((x_3)_H\big)^*\trb \big((x_1)_I\big)^*\Big) \# \big((x_4)_H\big)^* \otimes
\\
& \ \ \ \ \ \ \ \ \
\gamma(x_5^*) \Big( x_6^*\trb  \big((x_4)_A\big)^*\Big) \Big( x_7^*\trb \big( (x_1)_{II} (x_2)_A(x_3)_A  \big)^* \Big) \# x_8^*
&& [\textrm{using \eqref{a2}}]
\\
&=
\sum \gamma\Big(\big((x_2)_H\big)^*\Big) \Big(\big((x_3)_H\big)^*\trb \big((x_1)_I\big)^*\Big) \# (x_5^*)_H \otimes
\\
& \ \ \ \ \ \ \ \ \
\gamma(x_4^*) \big((x_5)_A\big)^* \Big( x_6^*\trb \big( (x_1)_{II} (x_2)_A(x_3)_A  \big)^* \Big) \# x_7^*
&& [\textrm{using \eqref{tres}}]
\\
&=
\sum \gamma\Big(\big((x_2)_H\big)^*\Big) \Big(\big((x_3)_H\big)^*\trb \big((x_1)_I\big)^*\Big) \# (x_6^*)_H \otimes
\\
& \ \ \ \ \ \ \ \ \
\gamma(x_4^*) \Big( x_5^*\trb \big( (x_1)_{II} (x_2)_A(x_3)_A  \big)^* \Big) \big((x_6)_A\big)^*  \# x_7^*
&& [\textrm{using \eqref{c7}}]
\\
&=
\sum \big(\gamma(x_1^*)\big)_1 (x_2^*)_I \# (x_3^*)_H \otimes \big(\gamma(x_1^*)\big)_2 (x_2^*)_{II} \# x_4^*
= \Delta(1\# x^*).
&& [\textrm{using \eqref{cuatro}}]
\end{align*}
Hence we have proved that $A\# H$ is a $*$-Hopf algebra.

\bigbreak

Conversely, let us assume that the formula \eqref{eq:estrella en bismash} endows $A^\psi\#_\chi H$ with a
$*$-bialgebra structure. Then from the equalities
\[
\big((1\# x)(a\# 1)\big)^*=(a\# 1)^*(1\# x)^* \quad \text{and} \quad
\big((1\# x)(1\# y)\big)^*=(1\# y)^*(1\# x)^*,
\]
we obtain \eqref{uno} and \eqref{dos}, respectively. Moreover from $(a\# x)^{**}=a\# x$, we deduce \eqref{cero}.
Finally, if we apply $\varepsilon\otimes\id\otimes\id\otimes\varepsilon$ and  $\id\otimes\varepsilon\otimes\id\otimes\varepsilon$ to
both sides of $\Delta\big((1\# x)^*\big)=\big(\Delta(1\# x)\big)^{*\otimes *}$, we deduce \eqref{tres} and \eqref{cuatro}.
\qed
\begin{obse}
For future use we present below a slightly different and equivalent version of the conditions of last theorem. 
\begin{align}
\sum \gamma\cc x_2 \dd\cc x_3\trb \gamma(x_1^*)^*\dd
&=
\varepsilon(x)1, 			        \label{vcero}\\
\sum \gamma\cc x_2 \dd\cc x_3\trb(x_1^*\trb a^*)^*\dd
&=
a\gamma(x), 			        \label{vuno}\\
\sum \gamma(y_1)\cc y_2\trb \gamma(x_1)\dd\chi(y_3,x_2)
&=
\sum \gamma(y_2x_2) \cc(y_3x_3)\trb\chi(x_1^*,y_1^*)^*\dd,		\label{vdos} \\
\sum x_{2H}\otimes\gamma(x_1)x_{2A}
&=
\sum \left({(x_1^*)}_H\right)^*\otimes \gamma(x_2)\left(x_3\trb\left({(x_1^*)}_A\right)^*\right), \label{vtres} \\
\sum \big(\gamma(x_1)\big)_1 x_{2I}\otimes \big(\gamma(x_1)\big)_2 x_{2II}
&=
\sum \gamma\cc\big((x_2^*)_H\big)^*\dd \Big(\big((x_3^*)_H\big)^*\trb \big((x_1^*)_I\big)^*\Big) \otimes \nonumber \\
&\ \ \ \ \  \ \ \
\gamma(x_4) \Big(x_5\trb \big((x_1^*)_{II}(x_2^*)_A(x_3^*)_A\big)^*\Big),
\label{vcuatro}
\end{align}
\end{obse}
\begin{obse}\label{rem:nico}
We consider two particular situations of special relevance. 
\begin{enumerate}
\item
In the hypothesis of the Theorem \ref{teo:* en AH}, if one takes the particular case that $\gamma(x)=\sum\chi^{-1}\cc x_2,\ant^{-1}(x_1)\dd$ --see \cite{kn:andrus2}--, then it can be proved that conditions:
\begin{align*}
(x\trb a)^* = \ant^{-1}(x^*)\trb a^* 
,\qquad
\chi(x,y)^* = \chi^{-1}\cc \ant^{-1}(x^*),\ant^{-1}(y^*)\dd		
\end{align*}
taken from \cite{kn:andrus2}, Proposition 3.2.9, imply \eqref{cero}, \eqref{uno} and \eqref{dos}.

\item If we consider  $\gamma(x)=\varepsilon(x)1$, then conditions \eqref{cero}--\eqref{cuatro} become:
\begin{align*}
\sum x_2^*\trb(x_1\trb a)^*
&=
a^*\varepsilon(x^*), 			        \\
\chi(x^*,y^*)
&=
\sum (x_2^*y_2^*)\trb\chi(y_1,x_1)^* ,		\\
\sum x^*_H\otimes x^*_A
&=
\sum \left(x_{1H}\right)^*\otimes \left(x_2^*\trb\left(x_{1A}\right)^*\right), \\
\sum x^*_I\otimes x^*_{II}
&=
\sum \big((x_{2H})^*\trb (x_{1I})^*\big) \otimes \big(x_3^*\trb (x_{1II}x_{2A})^*\big),
\end{align*}
for all $a\in A$, $x\in H$. These conditions, were the ones appearing in a preliminary version of this paper. 
\end{enumerate}
\end{obse}

\begin{teo} \label{teo:simplificar}
Let $(A,H,\trb,\rho,\chi,\psi)$ be a cocycle linked pair of Hopf algebras, where $A$ and $H$ are CQG and $\chi$ and $\psi$ are invertible.
We consider the map $\gamma:H\to A$ defined by $\gamma(x)=\sum\chi^{-1}\cc x_2,\ant^{-1}(x_1)\dd$, for all $x\in H$.
Let $\varphi_A, \varphi_H, \Phi$ be normal
integrals given as in Lemma \ref{lema:cocycle-bismash-coss}.
Assume that $\Phi$ is central, \textit{i.e.} it satisfies $\Phi((a \# x)(b \#y)) = \Phi((b \# y)(a \#x))$.
If \eqref{cero}-\eqref{cuatro} are verified, then $A^\psi\#_\chi H$ with the $*$-structure considered in \eqref{eq:estrella en bismash} is a CQG.
\end{teo}
\dem
Let $\langle -,- \rangle_{A^\psi\#_\chi H}$, $\langle -,- \rangle_A$ and $\langle -,- \rangle_H$ the hermitians forms corresponding
respectively to $\Phi$, $\varphi_A$ and $\varphi_H$.

From condition \eqref{a5} it follows that --see also equation \eqref{eqn:final}--: 
\[
\sum \big(x_1\trb \gamma(y_1)\big)\chi(x_2,y_2)=\sum\chi^{-1}\!\!\cc xy_2,\ant^{-1}(y_1)\dd,\qquad \forall x,y\in H.
\]
Also it is easy to see that
\[
(b\# y)^*(a\# x)=\sum \big(\gamma(y_1^*)\# y_2^*\big)\big(b^*a\# x\big),\qquad \forall x,y\in H,\ a,b\in A.
\]
Then
\begin{align*}
\langle a\# x,b\# y\rangle_{A\# H}
&=
\Phi\big((b\# y)^*(a\# x)\big)
=
\sum \Phi\Big( \big(\gamma(y_1^*)\# y_2^*\big)\big(b^*a\# x\big)\Big)
=
\sum \Phi\Big(\big(b^*a\# x\big) \big(\gamma(y_1^*)\# y_2^*\big)\Big) \\
&=
\sum \varphi_A \Big( b^*a \big(x_1\trb \gamma(y_1^*)\big) \chi(x_2,y_2^*) \Big) \varphi_H(x_3y_3^*)
=
\sum \varphi_A \Big( b^*a\, \chi^{-1}\!\!\cc xy_2^*,\ant^{-1}(y_1^*)\dd \Big) \varphi_H(x_3y_3^*) \\
&=
\sum \varphi_A \Big( b^*a\, \chi^{-1}\!\!\cc x_1y_2^* \varphi_H(x_2y_3^*),\ant^{-1}(y_1^*)\dd \Big)
=
\sum \varphi_A \Big( b^*a\, \chi^{-1}\!\!\cc 1 ,\ant^{-1}(y_1^*)\dd \Big) \varphi_H(xy_2^*) \\
&=
\sum \varphi_A ( b^*a ) \varphi_H(xy^*)
=
\langle a,b\rangle_{A} \langle y^*,x^*\rangle_{H}.
\end{align*}
The proof of the positivity of the inner product in $A^\psi\#_\chi H$ follows immediately from the above equality.
\fin

\begin{obse}\label{rem:finit-conm}
In the case that $A$ and $H$ are finite, the fact that $\Phi$ is central follows immediately from Theorem \ref{teo:semisimple}.
\end{obse}

\bigbreak

Next we study the particular case of this construction when $\chi$ and $\psi$ are trivial. 

\subsection{Linked pair of bialgebras}\label{subsection:linkedpairs}

In this section we consider the special case of a cocycle linked pair of bialgebras with trivial cocycle and cococycle,
in which case one can prove more precise results.

\begin{defi}
 A \emph{linked pair of bialgebras} is a cocycle linked pair of bialgebras with trivial cocycle and cococycle.
\end{defi}
\begin{obse}\label{obse:actioncoaction}
\begin{enumerate}
\item If we have a linked pair of bialgebras $(A,H,\trb,\rho)$,
where $A$ and $H$ are Hopf algebras, then conditions
\eqref{c1} and \eqref{c2}

are automatically verified (see \cite{kn:takeuchi}).

\item Clearly if $(A,H,\trb,\rho)$ is a linked pair of bialgebras, then  $(A,\trb)$ is a left
$H$-module algebra and $(H,\rho)$ is right $A$-comodule coalgebra. In this situation, the remaining compatibility relations
are \eqref{c1}, \eqref{c2}, \eqref{c5}, \eqref{c6} and \eqref{c7}.
\end{enumerate}
\end{obse}

\begin{obse}\label{rem:bi-algcoalg}
\begin{enumerate}
\item \label{rem:bi-algcoalg-1}
If $(H,A,\tl,\tr)$ is a matched pair of bialgebras and $A$ is finite dimensional, we have that $(A^\vee,H,\trb,\rho)$ becomes a
linked pair of bialgebras with the following structures:
\[
(x\trb \alpha)(a)=\alpha(a\tl x),\qquad \rho(x)=\sum x_H\otimes x_{A^\vee}\ \Leftrightarrow\
\sum x_H\, x_{A^\vee}(a)=a\tr x
\]
for all $a\in A$, $x\in H$ and $\alpha\in A^\vee$.
\item \label{rem:bi-algcoalg-2}
In view of Observation \ref{rem:action}.\ref{rem:action-2}, a linked pair
of bialgebras gives rise to a bialgebra $A^\psi\#_\chi H$, for which $\psi$ and $\chi$ are trivial.
The bialgebra obtained in this particular case is called the \emph{bismash} product of $A$ and $H$ and is
denoted by $A \# H$.
(See \cite{kn:takeuchi}, Theorem 6.2.2 or \cite{kn:Majid}.)
We have the following explicit formulae:
\[
(a\# x)(b\# y) = \sum a(x_1\trb b)\# x_2 y,\,\,
\Delta(a\# x) = \sum a_1\# {(x_1)}_H \otimes a_2 {(x_1)}_A\# x_2, \quad \forall a,b\in A,\ x,y\in H.
\]
In the case that $A$ and $H$ are Hopf algebras, then $A\# H$ is a Hopf algebra with antipode
\begin{equation*}
\ant(a\# x)=\sum \big(1\# \ant(x_H)\big)\big(\ant(a x_A)\# 1\big),\quad \forall  a\in A,\ x\in H.
\end{equation*}
Moreover, if we consider the natural inclusion and projection maps defined in Observation \ref{rem:ext}
then, $\iota_H$ is an algebra map and $\pi_A$ is a coalgebra map.
Hence $A\# H$ is generated as an algebra by the subalgebras $A\# 1$ and $1\# H$.
\end{enumerate}
\end{obse}

\begin{lema} \label{lema:intbismash}
Let $(A,H,\trb,\rho)$ be a linked pair of bialgebras.
If $A$ and $H$ are cosemisimple bialgebras with normal integrals $\varphi_A$ and $\varphi_H$ respectively,
then $A\# H$ is cosemisimple and $\Phi : A\# H \rightarrow \C$ defined by $\Phi(a\# x)=\varphi_A(a)\varphi_H(x)$ is a
normal integral. Moreover, in this situation $\varphi_H:H\to\C$ is a morphism of $A$-comodules.
\end{lema}
\dem
The first assertion follows from Lemma \ref{lema:cocycle-bismash-coss}.
The second assertion is proved by writing the right integral condition of $\Phi$ to $1\# x$ and then applying
$id_A \otimes \varepsilon_H$ to obtain:
\[
\sum\varphi_H\cc x_H\dd x_A=\varphi_H(x)1_A,\quad \forall x\in H.
\]
\qed

In connection with Observation \ref{rem:nico} we have the following.

\begin{lema} \label{lema:simpplificar}
Let $(A,H,\trb,\rho)$ be a linked pair of bialgebras, where $A$ and $H$ are $*$-Hopf algebras.
We consider $\gamma:H\to A$ defined by $\gamma(x)=\varepsilon(x)1$, for all $x\in H$.
Then condition \eqref{uno} is equivalent to
\begin{align} \label{compat-trb-*}
(x\trb a)^* &= \ant^{-1}(x^*)\trb a^*. 		
\end{align}
\end{lema}
\dem
If we assume \eqref{compat-trb-*}, then
\begin{align*}
\sum x_2^*\trb(x_1\trb a)^*
&=
\sum x_2^*\trb \left(\ant^{-1}(x_1^*)\trb a^*\right)
=
a^*\varepsilon(x^*).
\end{align*}

If we assume \eqref{uno}, then
\begin{align*}
(x\trb a)^*
&=
\sum \varepsilon(x_2^*)(x_1\trb a)^*
=
\sum \ant^{-1}(x_3^*)\trb \big( x_2^* \trb(x_1\trb a)^*\big)
=
\sum \ant^{-1}(x_2^*)\trb a^*\varepsilon(x_1^*)
=
\ant^{-1}(x^*)\trb a^*.
\end{align*}
\fin

\bigbreak

From Theorems \ref{teo:* en AH} and \ref{teo:simplificar}, Observation \ref{rem:finit-conm} and Lema \ref{lema:simpplificar} we get:

\begin{coro}
Let $(A,H,\trb,\rho)$ be a linked pair of bialgebras, where $A$ and $H$ are $*$-Hopf algebras.
Then the formula below:
\begin{equation*}
(a\# x)^*
=\sum x_1^*\trb a^*\# x_2^*,\quad \forall  a\in A,\ x\in H
\end{equation*}
defines a structure of $*$-Hopf algebra in $A\# H$ if and only if
\[
(x\trb a)^* = \ant^{-1}(x^*)\trb a^* \qquad \textrm{and} \qquad 			
\rho(x^*) = \sum \left(x_{1H}\right)^*\otimes \left(x_2^*\trb\left(x_{1A}\right)^*\right),
\]
for all $a\in A$, $x\in H$. In this situation, if $A$ and $H$ are finite CQG, then so is $A\# H$.
\fin
\end{coro}

In order to apply theorem \ref{teo:simplificar} without assuming that $A$ and $H$ are finite, we have the following.

\begin{prop}\label{teo:linkedcompacto} 
Let $(A,H,\trb,\rho)$ be a linked pair of bialgebras.
Let $\varphi_A$, $\varphi_H$, $\Phi$, be as in Lemma \ref{lema:intbismash}.
Then $\Phi$ is central if and only if $\varphi_A$ and $\varphi_H$ are central and $\varphi_A: A \rightarrow \C$ is a morphism of $H$--modules.

\end{prop}
\dem
We have $\Phi\big((a\# x)(b\# y)\big)=\sum\varphi_A\big(a(x_1\trb b)\big)\varphi_H(x_2y)$, then $\Phi$ is central if and only if
\begin{equation} \label{eq:central}
\sum\varphi_A\big(a(x_1\trb b)\big)\varphi_H(x_2y)
=
\sum\varphi_A\big(b(y_1\trb a)\big)\varphi_H(y_2x)
,\quad \forall  a,b\in A,\ x,y\in H.
\end{equation}
Assume  that $\varphi_A$ and $\varphi_H$ are central and $\varphi_A$ is a morphism of $H$--modules.
The last condition on $\varphi_A$ means that $\varphi_A(x\trb a) = \varepsilon(x)\varphi_A(a)$ for all $x \in H,\, a \in A$.
Then
\begin{align*}
\sum\varphi_A\big(a(x_1\trb b)\big)\varphi_H(x_2y)
&=
\sum\varphi_A\big((x_1\trb b)a\big)\varphi_H(x_2y)
=
\sum\varphi_A\Big((x_1\trb b)\big((x_2y_1)\trb a\big)\Big)\varphi_H(x_3y_2) \\
&=
\sum\varphi_A\Big((x_1\trb b)\big(x_2\trb (y_1\trb a)\big)\Big)\varphi_H(x_3y_2) \\
&=
\sum\varphi_A\Big(x_1\trb \big(b (y_1\trb a)\big)\Big)\varphi_H(x_2y_2)
=
\sum\varphi_A\big(b (y_1\trb a)\big)\varphi_H(xy_2) \\
&=
\sum\varphi_A\big(b (y_1\trb a)\big)\varphi_H(y_2x).
\end{align*}
Assume now that $\Phi$ is central. If we put $a=b=1$ in \eqref{eq:central} we get that $\varphi_H$ is central and if we put $x=y=1$ we get that $\varphi_A$ is central.

Using that $\varphi_H$ is central, if we put $b=1$ in \eqref{eq:central} we get
\begin{equation*}
\varphi_A(a)\varphi_H(yx)
=
\sum\varphi_A(y_1\trb a)\varphi_H(y_2x)
,\quad \forall  a\in A,\ x,y\in H.
\end{equation*}
Now, as $H$ is coFrobenius the validity of the above equation for all $x \in H$ implies that
\begin{equation*}
\varphi_A(a)y
=
\sum\varphi_A(y_1\trb a)y_2
,\quad \forall  a\in A,\ y\in H,
\end{equation*}
and applying $\varepsilon$ we obtain that $\varphi_A$ is a morphism of $H$--modules.
\fin

\section{The case of a cocycle Singer pair}
\label{section:singerpairs}
In this section we consider the general construction of a star for a cocycle linked pair of $*$--Hopf algebras in the case that $H$ is cocommutative and $A$ is commutative. 
Notice that in this situation the antipodes of $A$ and of $H$ are involutive.  

\subsection{General definitions}
\begin{defi}\label{defi:csingerpair} 
A cocycle linked pair of Hopf algebras $(A,H,\trb,\rho,\chi,\psi)$ is said to be a \textit{cocycle Singer pair} if $H$ is cocommutative, $A$ is commutative, 
and the cocycle and cococycle $\chi$ and $\psi$ are invertible.  
\end{defi}
The above definition is motivated by the fact that a linked pair of Hopf algebras $(A,H,\trb,\rho)$ in which  $A$ is a commutative Hopf algebra and $H$ is a
cocommutative Hopf algebra is called a \emph{Singer pair} (see for example \cite{kn:masuoka}).
\begin{obse}
\label{rem:singer}
\begin{enumerate}
\item Notice that if we have a cocycle Singer pair, then condition \eqref{c7} is authomatically verified,
and it is easy to see that we can remove $\chi$ and $\psi$ from relations \eqref{a4}, \eqref{b4}, \eqref{c5}, \eqref{c6}.
\item
Hence in this case being $H$ is cocommutative, $A$ commutative and $\chi$ and $\psi$ are convolution invertible, then
$(A,H,\trb,\rho,\chi,\psi)$ is a cocycle linked pair of bialgebras if and only if
$(A,H,\trb,\rho)$ is a linked pair of bialgebras and $\chi$ and $\psi$ verify relations
\eqref{a5}, \eqref{a6}, \eqref{b5}, \eqref{b6}, \eqref{c3}, \eqref{c4}, \eqref{c8}.
\end{enumerate}
\end{obse}

\begin{obse} 

In this case the conditions characterizing the structure become simpler.  
\begin{enumerate}
\item Conditions for $(\trb,\chi)$.
\[x\trb 1  = \varepsilon(x)1,\,\, x\trb ab = \sum (x_1\trb a)(x_2\trb b),\,\, 1\trb a = a,\]
\[x \trb (y \trb a)=
(xy)\trb a,\] 
\begin{equation}\sum \big( x_1\trb \chi(y_1,z_1)\big)\,\chi(x_2,y_2z_2)
=
\sum \chi(x_1,y_1)\, \chi(x_2 y_2,z), 
\chi(x,1)= \chi(1,x)=\varepsilon(x)1.\label{a41}
\end{equation}
for all $x,y,z\in H$, $a,b\in A$.
\item Conditions for $(\rho,\psi)$.
\[\sum\varepsilon(x_H)x_A =
\varepsilon(x)1,\,\, 
\sum x_{H1} \otimes x_{H2} \otimes x_A=
\sum x_{1H} \otimes x_{2H} \otimes x_{1A}x_{2A},
\sum x_H\varepsilon(x_A)=x\]
\[\sum x_{HH} \otimes x_{HA} \otimes x_A
=
\sum x_H \otimes x_{A1}  \otimes x_{A2},\]
\[\sum x_{1I1} x_{2HI} \otimes x_{1I2} x_{2HII}
\otimes x_{1II} x_{2A} =
\sum x_{1I} \otimes x_{1II1} x_{2I} \otimes  x_{1II2} 
x_{2II},\notag\]  
\begin{equation}\sum\varepsilon(x_I)x_{II}
=\sum x_I\varepsilon(x_{II})=\varepsilon(x)1,\label{b41}\end{equation}
for all $x\in H$.

\item Compatibility conditions for $(\trb,\ \rho,\ \chi,\ \psi)$. 
\[\varepsilon(x\trb a) = \varepsilon(x)\varepsilon(a), 
\sum 1_A\otimes 1_H = 1\otimes 1 ,
\varepsilon\big(\chi(x,y)\big) =
\varepsilon(x)\varepsilon(y),
\sum 1_I\otimes 1_{II}=
1\otimes 1 ,\]
\begin{align}
\sum (x\trb a)_1 \otimes (x \trb a)_2       
&=
\sum ({(x_1)}_H\trb a_1) \otimes {(x_1)}_A (x_2\trb a_2), \label{sc5}\\
\sum (xy)_H\otimes (xy)_A
&=
\sum {(x_1)}_H\, y_H \otimes {(x_1)}_A \big(x_2\trb y_A\big)
, \label{sc6}
\end{align}

\begin{equation}\label{eqn:sc7}
\Delta \chi \star \psi m = (\psi \otimes \varepsilon) \star \theta \star (1 \otimes \chi),
\end{equation} 
with 
 \begin{equation}\theta(x,y)= \sum (x_{1H1} \trb y_{1I})\chi(x_{1H2},y_{2H})\otimes x_{1A}(x_2 \trb (y_{1II}y_{2A}))\label{eqn:imp2}.
\end{equation}
\end{enumerate}
\end{obse}
\begin{obse}
\begin{enumerate}
\item Recall --see Observation \ref{obse:actioncoaction}-- that the above equations imply that the map $\trb: H \otimes A \rightarrow A$ is in fact an action and dually the map $\rho: H \rightarrow H \otimes A$ is a coaction. 
\item Condition \eqref{a2} implies that the convolution inverse of $\trb: H \otimes A \rightarrow A$ is:
\begin{equation}\label{eqn:invtriangle}
\trb^{-1}= \trb (\operatorname{id} \otimes \mathcal S).
\end{equation}
\item Similarly, condition \eqref{b2} yields the formula:
\begin{equation}\label{eqn:invrho}
\rho^{-1}= (\mathcal S \otimes \operatorname{id})\rho.
\end{equation}
\item If we apply $m(\mathcal S \otimes \operatorname{id})$ in equation \eqref{sc5} and use that the Hopf algebras are involutive, we obtain $\varepsilon(x)\varepsilon(a)1=\sum \mathcal S(x_{1H} \trb a_1)x_{1A}(x_2 \trb a_2)$ and using \eqref{eqn:invtriangle} we deduce that:
\begin{equation}\label{eqn:trianrho2}
x \trb a = \sum \mathcal S(x_H \trb {\mathcal S}(a))x_A\quad\textrm{or} \quad
\mathcal S(x \trb a) = \sum (x_H \trb {\mathcal S}(a))\mathcal S(x_A).
\end{equation}
\item Proceeding dually with equation \eqref{sc6} and $\rho$ we deduce that:
\begin{equation}\label{eqn:rhotrian}
\sum x_H \otimes x_A = \sum {\mathcal S}(\mathcal S (x_2)_H) \otimes (x_1 \trb \mathcal S (x_2)_A),
\end{equation}
or:
\begin{equation}\label{eqn:rhotrian2}
\sum \mathcal S(x_{2H}) \otimes (\mathcal S (x_1)\trb x_{2A}) = \sum \mathcal S(x)_H \otimes \mathcal S(x)_A,
\end{equation}
\end{enumerate}
\end{obse}
We list some additional properties needed for our computations. 
\begin{obse} The map $\gamma: H \rightarrow A, \,\gamma(x)=\sum \chi^{-1}(x_2,   {\mathcal S}^{-1}(x_1))$ considered in \cite{kn:andrus2} can also be written as: $\gamma(x)=\sum \chi^{-1}(x_2, {\mathcal S}(x_1))= \sum \chi^{-1}(x_1, {\mathcal S}(x_2))$. 
\begin{enumerate}
\item Clearly, $\gamma^{-1}(x)=\sum \chi(x_1, {\mathcal S}(x_2))=\sum \chi(x_2, {\mathcal S}(x_1))$. 
\item From the cocycle equation \eqref{a5} one easily obtains:  
\begin{equation}\label{eqn:important0}
x \trb \chi^{-1}(y,z)= \sum \chi(x_1,y_1z_1)\chi^{-1}(x_2y_2,z_2)\chi^{-1}(x_3,y_3).
\end{equation}

Using the above equation \eqref{eqn:important0} and applying it to the situation that $x \otimes y \otimes z$ is substituted by
$y \otimes x_2 \otimes {\mathcal S}(x_1)$, we obtain the formula: 
\begin{equation}\label{eqn:final}
y \trb \gamma(x)= \sum \chi^{-1}(y_1x_1,{\mathcal S}(x_2))\chi^{-1}(y_2,x_3).
\end{equation}
\item Similarly, from equation \eqref{a5} we deduce that
\begin{equation}\label{eqn:useful}
\sum x_1 \trb \gamma^{-1}(\mathcal S(x_2))= \gamma^{-1}(x).
\end{equation}
\end{enumerate}
\end{obse}
\begin{lema} \label{lema:preparation} In the above situation, the we have that:
\begin{equation}\label{eqn:moreimportant}
\sum y_1x_1 \trb \chi^{-1}(\mathcal S(x_2),\mathcal S(y_2))= \sum \gamma^{-1}(y_1x_1)\gamma(y_2)(y_3\trb \gamma(x_2))\chi(y_4,x_3).
\end{equation}
\end{lema}
\begin{proof}
By substitution of $x \otimes y \otimes z$ with $\sum y_1x_1 \otimes {\mathcal S}(x_2) \otimes {\mathcal S}(y_2)$ in equation \eqref{eqn:important0} we obtain: 
\begin{equation*}
\sum y_1x_1 \trb \chi^{-1}({\mathcal S}(x_2),{\mathcal S}(y_2))=\sum \chi(y_1x_1,{\mathcal S}(x_6){\mathcal S}(y_5))\chi^{-1}(y_2x_2{\mathcal S}(x_5),{\mathcal S}(y_4))\chi^{-1}(y_3x_3,{\mathcal S}(x_4))
\end{equation*} 
\begin{equation}\label{eqn:intermedia}= \sum \chi(y_1x_1,{\mathcal S}(x_4){\mathcal S}(y_5))\chi^{-1}(y_2,{\mathcal S}(y_4))\chi^{-1}(y_3x_3,{\mathcal S}(x_2))
\end{equation}
Hence, using \eqref{eqn:final} we deduce our result from the above formula \eqref{eqn:intermedia}.
\end{proof}

\begin{obse} \label{obse:compatibility} Next we extract some consequences of equation \eqref{c8}.
\begin{enumerate}

\item If we change variables in equation \eqref{eqn:19'} substituting $x \otimes y \mapsto \sum x_1 \otimes \mathcal S(x_2),$ the left hand side becomes: $$\sum(\Delta \chi \star \psi m)(x_1,\mathcal S(x_2))= \Delta \gamma^{-1}(x),$$
and we obtain the expression:
\begin{equation}\label{eqn:impo2}
\Delta\gamma^{-1}= \psi \star \nu \star (1 \otimes \gamma^{-1})\,\,\, \text{and} \,\,\, \psi^{-1} \star \Delta\gamma^{-1} \star (1 \otimes \gamma) =  \nu,\,\, \nu(x)= \sum \theta(x_1,\mathcal S(x_2)).  
\end{equation}
\item By sustitution in equation \eqref{eqn:imp}, we obtain the following expression for $\nu$:
\begin{equation*}
\nu(x)= \sum \big(x_{1H1} \trb \mathcal S(x_3)_I\big)\chi\big(x_{1H2}, \mathcal S(x_4)_H\big)\otimes x_{1A}\big(x_2 \trb \mathcal S(x_3)_{II} \mathcal S(x_4)_A\big).
\end{equation*}    

Applying equation \eqref{eqn:rhotrian2} we change $\nu$ into:

\end{enumerate}
\begin{align*}
\nu(x)&= \sum \big(x_{1H1} \trb \mathcal S(x_3)_I\big)\chi\big(x_{1H2}, \mathcal S(x_{5H})\big)\otimes x_{1A}\big(x_2 \trb \big(\mathcal S(x_3)_{II} \big(\mathcal S(x_4) \trb  x_{5A}\big)\big)\big)\\
&= \sum \big(x_{1H1} \trb \mathcal S(x_4)_I\big)\chi\big(x_{1H2}, \mathcal S(x_{6H})\big)\otimes x_{1A}\big(x_2 \trb \mathcal S(x_4)_{II}\big)\big(x_3\mathcal S(x_5) \trb x_{6A}\big)\\
&= \sum \big(x_{1H1} \trb (\mathcal Sx_3)_I\big)\chi\big(x_{1H2}, \mathcal S(x_{4H})\big)\otimes x_{1A}\big(x_2 \trb (\mathcal Sx_3)_{II}\big)x_{4A}\\
&=\sum \big(x_{1H1} \trb (\mathcal Sx_3)_I\big)\chi\big(x_{1H2}, \mathcal S(x_{2H})\big)\otimes x_{1A}x_{2A}\big(x_4 \trb (\mathcal Sx_3)_{II}\big).
\end{align*} 
Next if we apply the morphism $\Delta \otimes \operatorname{id} \otimes \operatorname{id}$ to equation 
\eqref{b2} and then substitute above, we obtain:
\begin{equation*}
\nu(x)= \sum \big(x_{1H1} \trb (\mathcal Sx_2)_I\big)\chi\big(x_{1H2}, \mathcal S(x_{1H3})\big)\otimes x_{1A}\big(x_3 \trb (\mathcal Sx_2)_{II}\big)
\end{equation*}
\begin{equation}\label{eqn:goodeq}
= \sum \big(x_{1H1} \trb (\mathcal Sx_2)_I\big)\gamma^{-1}(x_{1H2}) \otimes x_{1A}\big(x_3 \trb (\mathcal Sx_2)_{II}\big).\end{equation}
\end{obse}
\subsection{The structure of compact quantum group}

In the general situation of a cocycle linked pair each of whose components is equipped with a star structure, we have proved that for an arbitrary map $\gamma: H \rightarrow A$ if we define  $*:A^\psi\#_\chi H\rightarrow A^\psi\#_\chi H$ by the formula:
\begin{equation}
(a \# x)^{*} =\sum \gamma(x_1) (x_2^*\trb a^*) \# x_3^*;
\end{equation}
then conditions \eqref{cero}, \eqref{uno}, \eqref{dos}, \eqref{tres}, \eqref{cuatro} --or its versions \eqref{vcero}, \eqref{vuno}, \eqref{vdos}, \eqref{vtres}, \eqref{vcuatro}--are necessary and sufficient for $*$ to define a $*$--structure compatible with the bialgebra structure on $A^\psi\#_\chi H$. 

Next we prove that in the case of a cocycle Singer pair with the particular $\gamma$ considered in \cite{kn:andrus2}, i.e. $\gamma(x)= \sum \chi^{-1}(x_2,{\mathcal S}^{-1}(x_1))$ 
the above conditions are equivalent to those presented by Andruskiewitsch in \cite{kn:andrus2} with the numbers (3.2.1), (3.2.2), (3.2.6) and (3.2.7) that we write below as: 
\eqref{eqn:and1},\eqref{eqn:and2}, \eqref{eqn:and3} and \eqref{eqn:and4}, respectively.  

\begin{align}
(x \trb a)^* &= {\mathcal S}^{-1}(x^*) \trb a^*\label{eqn:and1}, \\
\chi(x,y)^*&= \chi^{-1}({\mathcal S}^{-1}(x^*),{\mathcal S}^{-1}(y^*))\label{eqn:and2},\\
\sum (x^*)_H \otimes (x^*)_A&=\sum \Big(\mathcal S\big({\mathcal S}^{-1}(x)_{H}\big)\Big)^*\otimes \big({\mathcal S}^{-1}(x)_A\big)^*\label{eqn:and3},\\
\sum (x^*)_I \otimes (x^*)_{II}&=\sum \big({\mathcal S}^{-1}(x)_{\widehat{I}}\big)^*\otimes \big({\mathcal S}^{-1}(x)_{\widehat{II}}\big)^*\label{eqn:and4}. 
\end{align}
\begin{obse}\label{obse:four}  Let $(A,H)$ be a cocycle Singer pair.
\begin{enumerate}
\item For a general $\gamma$, equation \eqref{vcero} is equivalent to the (convolution) invertibility of $\gamma$ and to the equality:
$\gamma^{-1}(x)=\sum x_1 \trb \gamma(x_2^*)^*$.

Indeed using the commutativity of $A$ and cocommutativity of $H$ we have:
\begin{equation}\sum \gamma(x_1)(x_2 \trb \gamma(x_3^*)^*)=\varepsilon(x)1=\sum (x_1 \trb \gamma(x_2^*)^*)\gamma(x_3).
\end{equation}
\item Notice also that for the particular $\gamma=\sum \chi^{-1}(x_2,{\mathcal S}^{-1}(x_1))$ using \eqref{eqn:useful}, from the condition above for $\gamma(x^*)^*$ and using that $\trb$ is an action, we deduce:
\begin{equation}\label{eqn:peculiar}
\gamma(x^*)^*=\gamma^{-1}(\mathcal Sx).
\end{equation}
\item In the case of a general $\gamma$ that is invertible \eqref{vuno} is equivalent to \eqref{eqn:and1}. Indeed from \eqref{vuno} we obtain --after cancelling $\gamma$-- that:
$\sum x_1 \trb (x_2^* \trb a^*)^* = \varepsilon(x)a$.

By acting with ${\mathcal S}(x)$ we deduce that: 
 
\begin{equation*}
 {\mathcal S}(x)\trb a = (x^* \trb a^*)^* \,\,\, \text{and} \,\,\, (x \trb a)^*= {\mathcal S}(x^*)\trb a^*.
\end{equation*}

\item We consider now the particular case that $\gamma(x)=\sum \chi^{-1}(x_2,{\mathcal S}^{-1}(x_1))$.
\begin{enumerate} 
\item Equation \eqref{vdos} is equivalent to \eqref{eqn:and2}. Indeed after some elementary manipulations it can be transformed into: 
\begin{equation}
\sum y_1x_1 \trb \chi(x_2^*,y_2^*)^*= \sum \gamma^{-1}(y_1x_1)\gamma(y_2)(y_3\trb \gamma(x_2))\chi(y_4,x_3)
\end{equation}

Then, the equality $\chi(x^*,y^*)^*=\chi^{-1}({\mathcal S}(x),{\mathcal S}(y))$ (or $\chi^{-1}(x^*,y^*)^*=\chi({\mathcal S}(x),{\mathcal S}(y))$ follows immediately from Lemma \ref{lema:preparation}, equation \eqref{eqn:moreimportant}. Clearly, this argument can be reversed.
\item If we start with equation \eqref{eqn:and2}, then  $\gamma(x^*)^* = \chi^{-1}(x_1^*,(\mathcal Sx_2)^*)^*=\chi(\mathcal Sx_1,x_2)$. Hence the equality $\gamma^{-1}(x)=\sum x_1 \trb \gamma(x_2^*)^*$ becomes $\chi(x_1,\mathcal Sx_2)=\sum x_1 \trb \chi(\mathcal Sx_2,x_3)$, that can be easily proved using \eqref{eqn:important0}. Hence, in the context of \cite{kn:andrus2}, condition \eqref{cero} is unnecessary. 
\end{enumerate}
\item Consider again the situation of a general $\gamma$. Once that the equality \eqref{eqn:and1} as well as the invertibility of $\gamma$ are guaranteed, we can prove that equation \eqref{vtres} is equivalent to \eqref{eqn:and3}. Indeed, in equation \eqref{vtres}, we can cancel $\gamma$ and obtain: $\sum (x^*)_H \otimes (x^*)_A = \sum (x_{1H})^* \otimes x_2^* \trb (x_{1A})^*= \sum (x_{1H})^* \otimes (\mathcal S(x_2) \trb x_{1A})^*$. 
Hence, to prove the equivalence, all we have to check is that \[\sum \mathcal S(\mathcal S(x)_H) \otimes \mathcal S(x)_A = \sum x_{1H} \otimes (\mathcal S(x_2) \trb x_{1A}), \]
and this is exactly equation \eqref{eqn:rhotrian2}. 
\item In the particular situation that $\gamma(x)= \sum \chi^{-1}(x_2,{\mathcal S}^{-1}(x_1))$ and using equation \eqref{eqn:impo2} one can write condition 
\eqref{cuatro} in the following manner:
\begin{equation}
\nu^{-1}(x) = \psi \star \Delta \gamma \star (1 \otimes \gamma^{-1})(x) = \sum \gamma(((x_2^*)_H)^*_1) (((x_2^*)_H)^*_2 \trb ((x_1^*)_I)^*) \otimes (x_3 \trb ((x_1^*)_{II}(x_2^*)_A)^*)\label{eqn:start}
\end{equation}
\end{enumerate}
\end{obse}
Next we proceed to compute $\nu(x^*)^*$ once we know that in our situation conditions \eqref{eqn:and1}--\eqref{eqn:and3} hold and assuming that $\gamma(x)= \sum \chi^{-1}(x_2,{\mathcal S}^{-1}(x_1))$.
\begin{lema}\label{lema:changenu} In the situation of a cocycle Singer pair and under the hypothesis of conditions \eqref{eqn:and1}--\eqref{eqn:and3}, and assuming that $\gamma(x)= \sum \chi^{-1}(x_2,{\mathcal S}^{-1}(x_1))$, we have that:
\begin{enumerate}
\item
\begin{equation}\label{eqn:nustar2} 
\nu(x^*)^*= \sum \gamma(\mathcal S(x_1)_{H1})\big(\mathcal S(x_1)_{H2} \trb (\mathcal S(x_3^*)_I)^*)\otimes \mathcal S(x_1)_{A}(\mathcal S(x_2) \trb (\mathcal S(x_3^*)_{II})^*).
\end{equation}
\item If we call $\mu(x)=\sum \gamma^{-1}(\mathcal S(x_1)_{H1})\big(\mathcal S(x_1)_{H2} \trb x_{3I}\big)\otimes \mathcal S(x_1)_{A}(\mathcal S(x_2) \trb x_{3II})$, then condition \eqref{cuatro} --see also condition \eqref{eqn:start}-- is equivalent to:
\begin{equation}\label{eqn:nuinvstar2}
\nu^{-1}(x^*)^*= \mu(x)
\end{equation}
\end{enumerate}
\end{lema}
\begin{proof}

\begin{enumerate}
\item Performing a direct substitution in equation \eqref{eqn:goodeq} we obtain:
\begin{equation}
\nu(x^*)^*= \sum \big((x^*)_{1H1} \trb \mathcal S(x_2^*)_I\big)^*\gamma^{-1}(x^*_{1H2})^* \otimes ((x^*)_{1A})^*\big(x^*_3 \trb \mathcal S(x_2^*)_{II}\big)^*.\end{equation}
Using equation \eqref{eqn:and3} in the equation above, we obtain:
\begin{equation*}
\nu(x^*)^*= \sum \big(\mathcal S(\mathcal S(x_{1})_H)_1^* \trb \mathcal S(x_2^*)_I\big)^*\gamma^{-1}(\mathcal S(\mathcal S(x_{1})_H)_2^*)^* \otimes \mathcal S(x_1)_A\big(x^*_3 \trb \mathcal S(x_2^*)_{II}\big)^*
\end{equation*}
\begin{equation*}=\sum \big(\mathcal S(\mathcal S(x_{1})_H)_2^* \trb \mathcal S(x_3^*)_I\big)^*\gamma^{-1}(\mathcal S(\mathcal S(x_{1})_H)_1^*)^* \otimes \mathcal S(x_1)_A\big(x^*_2 \trb \mathcal S(x_3^*)_{II}\big)^*.
\end{equation*}
Then, applying the equalities \eqref{eqn:peculiar} and \eqref{eqn:and1}, we deduce the required result. 
\item
 By a direct substitution we check that condition \eqref{eqn:start} is equivalent to:
\begin{equation*}\nu^{-1}(x^*)^* = \sum \gamma((x_{2H1})^*)^* ((x_{2H2})^* \trb (x_{1I})^*)^* \otimes 
(x_3^* \trb ((x_{1II}x_{2A})^*)^*
\end{equation*}
\begin{equation*}
= \sum \gamma^{-1}(\mathcal S(x_{2H1})) (\mathcal S(x_{2H2}) \trb x_{1I}) \otimes 
(\mathcal S(x_3) \trb (x_{1II}x_{2A}))
\end{equation*}
\begin{equation*}
= \sum \gamma^{-1}(\mathcal S(x_{2H1})) (\mathcal S(x_{2H2}) \trb x_{1I}) \otimes 
(\mathcal S(x_3) \trb x_{1II})(\mathcal S(x_4) \trb x_{2A})
\end{equation*}
\begin{equation}\label{eqn:cuasifinal}
= \sum \gamma^{-1}(\mathcal S(x_{2H1})) (\mathcal S(x_{2H2}) \trb x_{4I}) \otimes 
(\mathcal S(x_3) \trb x_{4II})(\mathcal S(x_1) \trb x_{2A}).
\end{equation}
Applying $\Delta \otimes \operatorname{id}$ to equation \eqref{eqn:rhotrian2} we obtain:
\begin{equation}
\sum \mathcal S(x_{2H1}) \otimes \mathcal S(x_{2H2}) \otimes (\mathcal S (x_1)\trb x_{2A}) = \sum (\mathcal S x)_{H1} \otimes (\mathcal S x)_{H2} \otimes (\mathcal S x)_A.
\end{equation}
By a direct substitution in equation \eqref{eqn:cuasifinal} we obtain equation \eqref{eqn:nuinvstar2}.
\end{enumerate}
\end{proof}
Next we prove the main result in this section.
\begin{theo}\label{theo:equivalence}
In the situation of a cocycle Singer pair with $\gamma(x)= \sum \chi^{-1}(x_2,{\mathcal S}^{-1}(x_1))$, conditions 
\eqref{eqn:and1},\eqref{eqn:and2},\eqref{eqn:and3},\eqref{eqn:and4} and conditions \eqref{cero},\eqref{uno},\eqref{dos},\eqref{tres},\eqref{cuatro} are equivalent.
\end{theo}
\begin{proof} In Observation \ref{obse:four} we took care of the equivalence between conditions 
\eqref{eqn:and1},\eqref{eqn:and2},\eqref{eqn:and3} and conditions \eqref{cero},\eqref{uno},\eqref{dos},\eqref{tres}. We need to consider the additional conditions \eqref{eqn:and4} and \eqref{cuatro}. Performing the convolution product of the right hand sides of equations \eqref{eqn:nustar2} and \eqref{eqn:nuinvstar2}, we have --for clarity we have abbreviated 
$\xi(x)=\sum q_I(x) \otimes q_{II}(x) = \sum (\mathcal S(x^*)_I)^* \otimes (\mathcal S(x^*)_{II})^*$ and $\psi(x)= \sum p_I(x) \otimes p_{II}(x) = \sum x_I \otimes x_{II}$:
\begin{equation*}
\sum \nu(x_1^*)^* \mu(x_2)= \sum \gamma(\mathcal S(x_1)_{H1})\big(\mathcal S(x_1)_{H2} \trb q_I(x_3))\gamma^{-1}(\mathcal S(x_2)_{H1})\big(\mathcal S(x_2)_{H2} \trb p_I(x_6)\big)
\end{equation*}
\begin{equation*}
\otimes \mathcal S(x_1)_{A}\mathcal S(x_2)_{A}(\mathcal S(x_4) \trb q_{II}(x_3))(\mathcal S(x_5) \trb p_{II}(x_6)\big)
\end{equation*}
\begin{equation*}
=\sum \gamma(\mathcal S(x_1)_{H1})\gamma^{-1}(\mathcal S(x_2)_{H1})\big(\mathcal S(x_1)_{H2} \trb q_I(x_3))\big(\mathcal S(x_2)_{H2} \trb p_I(x_5)\big)
\end{equation*}
\begin{equation}\label{eqn:almost}
\otimes \mathcal S(x_1)_{A}\mathcal S(x_2)_{A}(\mathcal S(x_4) \trb (q_{II}(x_3)p_{II}(x_5))\big).
\end{equation}
The equation \eqref{b2}, applied to $\mathcal S(x)$ yields:
\begin{equation}
\sum \mathcal S(x)_{H1} \otimes \mathcal S(x)_{H2} \otimes \mathcal S(x)_A
=
\sum \mathcal S(x_1)_H \otimes \mathcal S(x_2)_H \otimes \mathcal S(x_1)_A \mathcal S(x_2)_A,  
\end{equation}
and applying $\Delta \otimes \Delta \otimes \operatorname{id}$ to the above equation we get:
\begin{align*}
\sum \mathcal S(x)_{H1} \otimes \mathcal S(x)_{H2}\otimes \mathcal S(x)_{H3}\otimes \mathcal S(x)_{H4}\otimes \mathcal S(x)_A\\
=
\sum \mathcal S(x_1)_{H1}\otimes \mathcal S(x_1)_{H2} &\otimes \mathcal S(x_2)_{H1} \otimes \mathcal S(x_2)_{H2} \otimes \mathcal S(x_1)_A \mathcal S(x_2)_A.  
\end{align*}
By substitution in equation \eqref{eqn:almost}
we obtain:
\begin{equation*}
\sum \gamma(\mathcal S(x_1)_{H1})\gamma^{-1}(\mathcal S(x_2)_{H1})\big(\mathcal S(x_1)_{H2} \trb q_I(x_3))\big(\mathcal S(x_2)_{H2} \trb p_I(x_5)\big)
\end{equation*}
\begin{equation*}
\otimes \mathcal S(x_1)_{A}\mathcal S(x_2)_{A}(\mathcal S(x_4) \trb (q_{II}(x_3)p_{II}(x_5))\big)
\end{equation*}
\begin{equation*}
=\sum \gamma(\mathcal S(x_1)_{H1})\gamma^{-1}(\mathcal S(x_1)_{H2})\big(\mathcal S(x_1)_{H3} \trb q_I(x_2))\big(\mathcal S(x_1)_{H4} \trb p_I(x_4)\big)
\end{equation*}
\begin{equation*}
\otimes \mathcal S(x_1)_{A}(\mathcal S(x_3) \trb (q_{II}(x_2)p_{II}(x_4))\big)
\end{equation*}
\begin{equation*}
=\sum \mathcal S(x_1)_{H} \trb (q_I(x_3)p_I(x_4))
\otimes \mathcal S(x_1)_{A}(\mathcal S(x_2) \trb (q_{II}(x_3)p_{II}(x_4))\big)
\end{equation*}
\begin{equation}\label{eqn:aunpaso}
= \sum \pi_{\mathcal Sx_1}(q_I(x_2)p_I(x_3)\otimes q_{II}(x_2)p_{II}(x_3)) = \sum \pi_{\mathcal Sx_1}((\xi \star \psi)(x_2)),
\end{equation}
Where the map $\pi_x: A \otimes A \rightarrow A \otimes A$ is defined as follows: $\pi_x(a \otimes b)= \sum (x_{1H} \trb a) \otimes x_{1A}(x_2 \trb b)$. 
We have proved that:
\[\sum \nu(x_1^*)^* \mu(x_2)= \sum \pi_{\mathcal Sx_1}((\xi \star \psi)(x_2))\]
An elementary computation shows that $\sum \pi_{x_1}\pi_{Sx_2} (a \otimes b)= a \otimes b$, and also that $\pi_x(1 \otimes 1)= \varepsilon(x) 1 \otimes 1$.  

Then, $\mu(x)=\nu^{-1}(x^*)^*$ if and only if $\sum \pi_{\mathcal Sx_1}((\xi \star \psi)(x_2))= \varepsilon(x) 1 \otimes 1$, i.e. if and only if $(\xi \star \psi)(x)=\varepsilon(x)1 \otimes 1 $, i.e. if and only if $\xi = \psi^{-1}$, that is exactly condition \eqref{eqn:and4} --see also Lemma \ref{lema:changenu}. 

\end{proof}
\subsection{Matched pair of groups.}
In the special case of a pair of Hopf algebras of the form $A=\mathbb C^ G$ and $H=\mathbb C F$, where 
$F$ and $G$ are finite groups, a cocycle Singer pair can be produced directly at the level of the groups $F$ and $G$ by
enriching them with four maps $\tl$, $\tr$, $\sigma$ and $\tau$ that we describe below. This construction
has been presented in \cite{kn:kac} and further studied in \cite{kn:masuoka}.

Using this description and the results of the general case, we give necessary compatibility conditions between the four
maps mentioned above, for the product Hopf algebra to be a $*$-Hopf algebra and a CQG.

\begin{defi}\label{defi:matchedgroups}
A \emph{matched pair of groups} (\cite{kn:takeuchi}) is a quadruple $(F,G,\tl,\tr)$ where $F$ and $G$ are
groups and
\[
\begin{array}{c}
G \stackrel{\triangleleft}{\leftarrow} G\times F\stackrel{\triangleright}{\to}F,\\
g\triangleleft f \leftarrow (g,f) \to g\triangleright f
\end{array}
\]
are actions of the groups $F$ and $G$ on the {\em sets}
$G$ and $F$ respectively, satisfying the conditions that follow:
\begin{align}
g\tr ff' &= (g\tr f)\big((g\tl f)\tr f'\big), 			\label{compat-lr-1}\\
gg'\tl f &= \big(g\tl (g'\tr f)\big)(g'\tl f), 			\label{compat-lr-2}
\end{align}
for all $f,f'\in G$ and $g,g'\in F$.
\end{defi}

\begin{obse}
\begin{enumerate}
\item It is easy to prove for a matched pair of groups that $1 \tl f=1$ and $g \tr 1 = 1$, $\forall f \in F, g \in G$.
\item
If $(F,G,\tl,\tr)$ is a matched pair of groups, then we define a product in the set $F\times G$ by
\[
(f,g)(f',g')=(f(g\tr f'),(g\tl f')g'),\quad \forall g,g'\in F,\ f,f'\in G.
\]
$F\times G$ with this product is a group that we call $F\bowtie G$.
If we apply the usual functor between groups and Hopf algebras sending a group to its group algebra and extend $\tl$
and $\tr$ in the obvious manner, then it is easy
to see that  $(F,G,\tl,\tr)$ is a matched pair of groups if and only if $(\C F,\C G,\tl,\tr)$ is a matched pair of Hopf
algebras and we have $\C F\bowtie \C G\cong \C (F\bowtie G)$.
\item
If $(F,G,\tl,\tr)$ is a matched pair of groups where $G$ is a finite group, it follows from Observation 
\ref{rem:bi-algcoalg} (1), 

that the quadruple $\left(\C^G,\C F,\trb,\rho\right)$ is a Singer pair.
Indeed, given the arrows
$G \stackrel{\triangleleft}{\leftarrow} G\times F\stackrel{\triangleright}{\to}F$, we define
$\C F \stackrel{\rho}{\to} \C F\otimes \C^G\stackrel{\trb}{\to} \C^G$ by:
\[
f\trb e_g = e_{g\triangleleft f^{-1}},\qquad \rho(f)=\sum_{g\in G} g\triangleright f\otimes e_g,
\qquad \forall f\in F,\ g\in G,
\]
where we have denoted as $\{e_g : g \in G\}$ the fundamental idempotents in $\C^G$.
\end{enumerate}
\end{obse}

The above structure can be enriched by adding a cocycle $\chi: \C F \otimes \C F \to \C ^G$ and a cocoycle
$\psi: C^F \to \C ^G\otimes \C ^G$ in order to obtain a cocycle Singer pair.
Any pair of convolution invertible linear maps:
\[
\begin{array}{c}
\C F\otimes \C F \stackrel{\chi}{\to} \C^G\\
f\otimes f' \mapsto \chi(f,f')
\end{array}
\quad
\begin{array}{c}
\C F\stackrel{\psi}{\to} \C^G \otimes \C^G\\
f \mapsto \sum f_I\otimes f_{II}
\end{array}
\]
can be described in terms of the natural basis as two families of functions:
\[
\begin{array}{c}
G\times F\times F\stackrel{\sigma}{\to} \C^\times\\
(g,f,f') \mapsto \sigma(g;f,f')
\end{array}
\quad
\begin{array}{c}
G\times G\times F\stackrel{\tau}{\to} \C^\times\\
(g,g',f) \mapsto \tau(g,g';f)
\end{array}
\]
such that:
\[
\chi(f,f')=\sum_{g\in G}\sigma(g;f,f')e_g,\qquad \sum f_I\otimes f_{II}= \sum_{g,g'\in G}\tau(g,g';f)e_g\otimes e_{g'},
\qquad \forall f, \, f' \in F.
\]

Note that the invertibility of $\chi$ and $\psi$ is equivalent to the fact that $\sigma$ and $\tau$ take non zero values.

We write $\C^G\#_{\sigma,\tau}\C F$ for the vector space $\C^G\otimes\C F$ and
$e_g\# f$ for the element $e_g\otimes f$ of the standard basis.

\begin{obse}\label{teo:masuoka-Lemma-1.2}
In this situation, Observation \ref{rem:action}.\ref{rem:action-2}, Theorem \ref{teo:ext-hopf} and Observation \ref{rem:singer}, imply Lemma 1.2 in \cite{kn:masuoka}. 
Explicitly: the vector space $\C^G\#_{\sigma,\tau}\C F$ with the product, coproduct, unit and counit defined below
\begin{align*}
(e_g\# f)(e_{g'}\# f') &= \delta_{g\tl f,g'}\,\sigma(g;f,f')\, e_g\# ff',		
\\
\Delta(e_g\# f)
&= \sum_{g'g''=g}\tau(g',g'';f)\, e_{g'}\#(g''\tr f)\otimes e_{g''}\# f,
1_{\C^G\#_{\sigma,\tau}\C F} = \sum_g e_g\# 1, & \hspace*{.5cm} \varepsilon_{\C^G\#_{\sigma,\tau}\C F}(e_g\# f)= \delta_{g,1},	
\end{align*}
is a bialgebra if and only if $(F,G,\tl,\tr)$ is a matched pair of groups and $\sigma$ and $\tau$ verify the following conditions:
\begin{align}
\sigma(g\tl f;f',f'')\, \sigma(g;f,f'f'') &= \sigma(g;f,f')\, \sigma(g;ff',f''), \label{sigma-1}\\
\sigma(1;f,f')&= \sigma(g;1,f')=\sigma(g;f,1)=1,			\label{sigma-2} \\
\tau(gg',g'';f)\, \tau(g,g';g''\tr f) &= \tau(g',g'';f)\, \tau(g,g'g''; f), \label{tau-1}\\
\tau(1,g';f)&= \tau(g,1;f)=\tau(g,g';1)=1,			\label{tau-2} \\
\sigma(gg';f,f')\,\tau(g,g';ff') &=
\sigma(g;g'\tr f,(g'\tl f)\tr f')\, \sigma(g';f,f')\,
\tau(g,g';f)\, \tau(g\tl(g'\tr f),g'\tl f;f'), 			\label{compat-sigma-tau}
\end{align}
for all $g,g',g''\in G,\ f,f',f''\in F$. Moreover, $\C^G\#_{\sigma,\tau}\C F$ is a Hopf algebra with antipode:
\begin{align}
\ant(e_g \# f) &=
\sigma\left(g^{-1};g\triangleright f, (g\triangleright f)^{-1}\right)^{-1}
\tau\left(g^{-1},g;f\right)^{-1}e_{(g\triangleleft f)^{-1}}\#(g\triangleright f)^{-1}. 	\label{antipoda}
\end{align}
\end{obse}
A pair $(\sigma, \tau)$ as above, verifying conditions

\eqref{sigma-1}--\eqref{compat-sigma-tau},
is called a pair of \textit{compatible normal cocycles}.

\begin{teo} \label{teo:mp-cqg}
Let $(F,G,\tl,\tr,\sigma,\tau)$ be such that $(F,G,\tl, \tr)$ is a matched pair of groups and $\sigma, \tau$ are
compatible normal cocycles. Let $\alpha:F\times G\to \C$ be an arbitrary map.
\begin{enumerate}
\item
The formula
\begin{align}
(e_g\# f)^\ast=\alpha\cc f^{-1},g\tl f\dd e_{g\tl f}\# f^{-1}
\label{eq:alfa}
\end{align}
defines a structure of $*$-Hopf algebra in $\C^G\#_{\sigma,\tau}\C F$ if and only if
\begin{align}
\alpha(1,g) 
= 
\alpha(f,1)&=1, \label{cond-0}\\
\alpha(f,g)\overline{\alpha\cc f^{-1},g\tl f\dd}
&=
1,  \label{cond-1}\\
\alpha(f_1,g)\alpha (f_2,g\tl f_1) \sigma(g;f_1,f_2) 
&=
\alpha(f_1f_2,g) \overline{\sigma\!\cc g\tl(f_1f_2);f_2^{-1},f_1^{-1}\dd}, \label{cond-2} \\	
\alpha(f,g_1g_2) \tau(g_1,g_2;f) 
&=
\alpha(g_2\tr f,g_1) \alpha(f,g_2) \overline{\tau\!\left( g_1\tl(g_2\tr f),g_2\tl f,f^{-1} \right)}, \label{cond-3} 
\end{align}
for all $g,g_1,g_2\in G,\ f,f_1,f_2\in F$.
\item
The pair $\left(\C^G\#_{\sigma,\tau}\C F,*\right)$ is a CQG if and only if $\sigma$ and
$\tau$ verifiy \eqref{cond-0}-\eqref{cond-3} and
\begin{align}
\alpha\!\cc f^{-1},g\tl f\dd\sigma\!\left(g;f,f^{-1}\right) &> 0,	\label{cond-compacta}
\end{align}
for all $g\in G,\ f\in F$.
\end{enumerate}
\end{teo}
\dem
For the first assertion, we recall that $\C^G$ and $\C F$ are CQG with respect to the $\ast$-structures:
\[
(e_g)^\ast = e_g,\quad f^\ast=f^{-1},\quad \forall g\in G,\ f\in F.
\]
If we define a linear map $\gamma:\C F\to \C^G$ by $\gamma(f)=\sum_{g\in G} \alpha(f,g)e_g$, for all $f\in F$, then the formula \eqref{eq:estrella en bismash} becomes 
\eqref{eq:alfa}. Moreover, \eqref{cond-0} is equivalent to $\gamma(1)=1$ and $\varepsilon\gamma=\varepsilon$, \eqref{uno} and \eqref{tres} are automatically verified, 
and conditions \eqref{cond-1}, \eqref{cond-2} and \eqref{cond-3} are equivalent to conditions \eqref{cero}, \eqref{dos} and \eqref{cuatro}, respectively.
Hence the first assertion follows from Theorem \ref{teo:* en AH}.  

\bigbreak

For the second assertion, we observe that the linear map
$\Phi : \C^G\#_{\sigma,\tau}\C F \rightarrow \C$ defined by
\(
\Phi(e_g\# f)=\frac{1}{|G|}\delta_{f,1}
\)
is a normal integral and
\[
\langle e_{g_1}\# f_1, e_{g_2}\# f_2\rangle_\Phi=\delta_{g_1,g_2}\,\delta_{f_1,f_2}\,\frac{1}{|G|}\,
\alpha\!\cc f_1^{-1},g_1\tl f_1\dd\sigma\!\left(g_1;f_1,f_1^{-1}\right)
,\quad \forall g_1,g_2\in G,\ f_1,f_2\in F.
\]
Then $\langle \ , \ \rangle_\Phi$ is positive definite if and only if \eqref{cond-compacta} follows.
\qed

\begin{obse}\label{obsesivo}
We consider two particular cases of the above theorem.

If we take $\alpha(f,g)=\sigma\!\cc g;f,f^{-1}\dd^{-1}$ which corresponds to $\gamma(x)=\sum\chi^{-1}\cc x_2,\ant^{-1}(x_1)\dd$, then conditions \eqref{cond-0}-\eqref{cond-3} are equivalent to
\[
|\sigma(g_1;f_1,f_2)|=|\tau(g_1,g_2;,f_1)|=1,\qquad \forall g_1,g_2\in G,\ f_1,f_2\in F,
\]
and condition \eqref{cond-compacta} is automatically verified.
Now if we take the trivial map $\alpha(f,g)=1$ which corresponds to $\gamma(x)=\varepsilon(x)1$, then conditions \eqref{cond-0}-\eqref{cond-compacta} are equivalent to
\[
\sigma\left(g\tl ff';f'^{-1},f^{-1}\right) =\overline{\sigma(g;f,f')},\quad
\tau\left(g\tl(g'\tr f),g'\tl f,f^{-1}\right) =\overline{\tau(g,g';f)},\quad
\sigma\left(g;f,f^{-1}\right) > 0,
\]
for all $g,g'\in G,\ f,f'\in F$.
\end{obse}

\subsection{Examples}\label{subsection:ejemplos}

In this subsection we present two examples, where the first one is a parametric family of CQG that generalizes an example due to Masuoka \cite{kn:masuoka3}.

Both examples are based on the same matched pair of groups.

Let $n$ be an integer greater than one and $C_n$ be the cyclic group of order $n$.
We consider  $F=C_2=\{1,x\}$ and $G=C_n\times C_n=\left\{a^ib^j:\ i,j\in\Z_n\right\}$, where $a,b$ are generators of $C_n$.
The group $F$ acts on $G$ by group automorphisms with the following action:
$$ a^ib^j \tl x=a^ib^{-j},\qquad \forall i,j\in\Z_n.
$$
Hence if we consider the trivial action $\tr:G\times F\to F$, then $(G,F,\tl,\tr)$ is a matched pair of groups and
$G\bowtie F$ is the semidirect product of $G$ and $F$.

\begin{example} \label{ej:6.15}
If we take $\sigma$ as the trivial map --that produces the trivial cocycle-- and we define $\tau$ by
\[
\tau\cc a^ib^j,a^kb^l;1\dd=1,\qquad \tau\cc a^ib^j,a^kb^l;x\dd=\zeta^{jk} \eta^{il}. \qquad \forall i,j,k,l\in\Z_n.
\]
where $\zeta,\eta$ are complex numbers with $|\zeta|=|\eta|=1$, then $(\sigma,\tau)$ is a pair of compatible normal cocycles and therefore 
we have a Hopf algebra extension $\C^{C_n \times C_n }\to \C^{C_n \times C_n}\#_{\sigma,\tau}\C C_2\to \C C_2$.

The only solution of equations \eqref{cond-0}-\eqref{cond-compacta} is the trivial one, so $\C^{C_n \times C_n}\#_{\sigma,\tau}\C C_2$ is a CQG by defining
\[
\cc e_{ij}\# 1\dd^*=e_{ij}\# 1,\quad \cc e_{ij}\# x\dd^*= e_{i,-j}\# x,\quad \forall i,j\in\Z_n.
\]
\end{example}

\begin{example} \label{ej:6.16}

We change slightly our perspective and as both groups are abelian we view the matched pair
\begin{align*}
C_n \times C_n \stackrel{\triangleleft}{\longleftarrow} (C_n \times C_n&) \times
C_2 \stackrel{\triangleright}{\longrightarrow}C_2,\\
a^ib^{-j}=a^ib^j\triangleleft x \longleftarrow (a^ib^j&,x) \longrightarrow a^ib^j \triangleright x = x
\end{align*}
as actions on the other side:
\begin{align*}
C_2 \stackrel{\leftharpoonup}{\longleftarrow} C_2 \times (&C_n \times C_n) \stackrel{\rightharpoonup}{\longrightarrow}
C_n \times C_n,\\
x=x \leftharpoonup a^ib^j \longleftarrow (x&,a^ib^j) \longrightarrow x \rightharpoonup a^ib^j = a^ib^{-j}.
\end{align*}
Now we take $\tau$ as the trivial map --that produces the trivial cococycle--, and we define $\sigma$ by
\[
\sigma\cc 1; a^ib^j,a^kb^l \dd=1,\qquad \sigma\cc x; a^ib^j,a^kb^l\dd=\zeta^{il}\eta^{-jk}, \qquad \forall i,j,k,l\in\Z_n,
\]
where $\zeta,\eta$ are complex numbers with $|\zeta|=|\eta|=1$, then we have a Hopf algebra extension \\
$\C^{C_2}\to \C^{C_2}\#_{\sigma,\tau}\C [C_n \times C_n]\to \C [C_n \times C_n]$ with similar
properties than before.

\bigbreak
We obtain a solution of equations \eqref{cond-0}-\eqref{cond-compacta} by defining $\alpha:(C_n\times C_n)\times C_2\to\C$ by
$\alpha(a^ib^j,1)=1$ and $\alpha(a^ib^j,x)=\cc\frac{\zeta}{\eta}\dd^{ij}$, for all $i,j\in\Z_n$. Then
$\C^{C_2}\#_{\sigma,\tau}\C [C_n \times C_n]$ is a CQG by defining 
\[
\cc e_1\# a^ib^j\dd^*=e_1\# a^{-i}b^{-j},\quad \cc e_x\# a^ib^j\dd^*=\cc\frac{\zeta}{\eta}\dd^{ij}e_x\# a^{-i}b^{-j},\quad  \forall i,j\in\Z_n,
\]
\end{example}

\end{document}